\documentclass[10pt,a4paper]{amsart}
\usepackage{amssymb,amsthm,amsmath,mathrsfs}

\usepackage[T1]{fontenc}
\usepackage[utf8]{inputenc}
\usepackage[british]{babel}
\usepackage{geometry}
\usepackage{color}

\newcommand{\C}{\mathbb{C}}

\newcommand{\R}{\mathbb{R}}

\newcommand{\N}{\mathbb{N}}
\newcommand{\Z}{\mathbb{Z}}

\newcommand{\cF}{\mathcal{F}}
\newcommand{\cC}{\mathcal{C}}

\newcommand{\cH}{\mathcal{H}}

\newcommand{\G}{\Gamma}
\newcommand{\g}{\gamma}
\newcommand{\tv}{\rightarrow}

\newtheorem{theorem}{Theorem}[section]
\newtheorem{lemme}[theorem]{Lemma}
\newtheorem{prop}[theorem]{Proposition}
\newtheorem{corollaire}[theorem]{Corollary}
\newtheorem{definition}[theorem]{Definition}
\newtheorem{fait}[theorem]{Fact}

\DeclareMathOperator{\Id}{Id}

\newcommand{\bP}{\mathbb{P}}

\newcommand{\fa}{\mathfrak{a}}
\newcommand{\fg}{\mathfrak{g}}

\newcommand{\fn}{\mathfrak{n}}

\newcommand{\cU}{\mathcal{U}}
\newcommand{\cE}{\mathcal{E}}
\newcommand{\cB}{\mathcal{B}}

\newcommand{\cV}{\mathcal{V}}
\newcommand{\cW}{\mathcal{W}}
\newcommand{\ccF}{\mathcal{F}}

\title[Topological mixing]{Topological mixing of the Weyl Chamber flow}
\author{Nguyen-Thi Dang, Olivier Glorieux }
\begin{document}
\maketitle

\begin{abstract}
In this paper, we study topological properties of the right action by translation of the Weyl Chamber flow on the space of Weyl chambers.
We obtain a necessary and sufficient condition for topological mixing.
\end{abstract}

(\footnote{MSC Classification 54H20, 37B99, 53C30, 58E40 Secondary 53C35})

\section{Introduction}
Let $G$ be  semisimple real, connected, Lie group of non compact type. 
Let $K$ be a maximal compact subgroup of $G$ and $A$ a maximal torus of $G$ for which there is a Cartan decomposition. 
Let $M$ be the centralizer of $A$ in $K$.
We establish mixing properties for right action by translation of one parameter subgroups of $A$ on quotients $\G \backslash G /M$ where $\G$ is a discrete, Zariski dense subgroup of $G$.
\vspace{5pt}

The particular case when $G$ is of real rank one is well known.
In this case, the symmetric space $X=G/K$ is a complete, connected, simply connected Riemannian manifold of negative curvature. 
The right action by translation of $A$ on $G/M$ coincides with the geodesic flow on $T^1 X$.
Dal'bo \cite{dalbo2000feuilletage} proved that it is mixing (on its nonwandering set) if and only if the length spectrum is non arithmetic.
The latter holds when $\G$ is a Zariski dense subgroup, 
see Benoist \cite{benoist2000proprietes},  Kim \cite{kim2006length}.
\vspace{5pt}

We are interested in cases where $G$ is of higher real rank $k\geq 2$.
When $\G \backslash G /M$ is of finite volume, i.e. when $\G$ is a lattice, it follows from Howe-Moore's Theorem that the action of any noncompact subgroup of $G$ is mixing.

We study the general situation of any discrete, Zariski dense subgroup, which of course includes the case of lattices. 
\vspace{5pt}

If $\G \backslash G/M$ has infinite volume, the  known results are not as general. 

In the particular case of so-called Ping-Pong subgroups of $\mathrm{PSL}(k+1,\R)$,  Thirion \cite{thirion2007sous}, \cite{thirion2009proprietes} proved mixing with respect to a natural measure on $\Omega(X)$ for a one parameter flow associated to the "maximal growth vector" introduced by Quint in \cite{quint2002divergence}.
Sambarino \cite{sambarino2015orbital} did the same for hyperconvex representations.

Finally, Conze-Guivarc'h in \cite{conze-guivarch-densite2000} proved for any Zariski dense subgroup $\G$, the topological transitivity (i.e. existence of dense orbits) of the right $A-$action on a natural closed $AM-$invariant set $\Omega(X) \subset \G \backslash G /M$.

\vspace{5pt}

Let $\fa\simeq \R^k$ be the Cartan Lie subalgebra over $A$ and $\fa^{++}$ the choice of a positive Weyl chamber.
For any $\theta\in \fa^{++}$, the Weyl chamber flow $(\phi_t^{\theta})$ corresponds to the right action by translation of $\exp(t\theta)$. 
Benoist \cite{benoist1997proprietes} introduced a convex limit cone $\cC (\G)\subset \fa$ and proved that for Zariski dense semigroups, the limit cone is of non empty interior.
We prove topological mixing for any direction of the interior of $\cC (\G)$.

\begin{theorem}\label{th-main_geom}
Let $G$ be a semisimple, connected, real linear Lie group, of non-compact type.
Let $\G$ be a Zariski dense, discrete subgroup of $G$.
Let $\theta \in \fa^{++}$.

Then the dynamical system ($\Omega(X), \phi_t^\theta)$ is topologically mixing if and only if $\theta$ is in the interior of the limit cone $\cC (\G)$. 
\end{theorem}

Taking $\widetilde{\Omega}\subset G/M$ to be the universal cover of $\Omega(X)$, we remark that this Theorem is a direct consequence of the following statement, where $\G$ is a Zariski dense semigroup of $G$.
We insist that under this hypothesis, $\G$ is not necessarily a subgroup and can even be non discrete. 

\begin{theorem}\label{th-main}
Let $G$ be a semisimple, connected, real linear Lie group, of non-compact type.
Let $\G$ be a Zariski dense semigroup of $G$.
Let $\theta \in \fa^{++}$.

Then $\theta$ is in the interior of the limit cone if and only if for all nonempty open subsets $\widetilde{U},\widetilde{V}\subset \widetilde{\Omega}(X) $, there exists $T> 0$ so that for any later time $t>T$, there exists $\g_t \in \G$ with
			$$\g_t \widetilde{U} \cap \phi_{t}^\theta ( \widetilde{V})\neq \emptyset .$$
\end{theorem}

In the first section, we give some background on globally symmetric spaces. We introduce the space of Weyl chambers, the Weyl chamber flow, give a compactification of the space of Weyl chambers and present a higher rank generalization of the Hopf coordinates.

In the second section, we introduce the main tools: Schottky semigroups and estimations on the spectrum of products of elements in $G$. 

In the third section, we introduce the \emph{non-wandering Weyl chambers set}, it is a closed $AM-$invariant subset $\Omega(X)\subset\G \backslash G /M$. 
Then we study topological transitivity in Proposition \ref{prop_transitivité}.
We prove that if the flow $\phi_t^\theta$ is topologically transitive in $\Omega(X)$, where $\theta\in \fa^{++}$, then the direction $\theta$ must be in the interior of the limit cone.
Since topological mixing implies topological transitivity, this
provides one direction of the main Theorem \ref{th-main}.

In the last section, we prove a key Proposition \ref{prop-moving along fa} using density results that come from non-arithmeticity of the length spectrum. Then we prove the main theorem.

In the appendix we prove a density lemma of subgroups of $\R^n$ needed in the proof of Proposition \ref{prop-moving along fa}. 

\begin{center}
\fbox{
\begin{minipage}{0.7\textwidth}
In the whole article, $G$ is a semisimple, connected, real linear Lie group, of non-compact type.
\end{minipage}}
\end{center}

\section{Background on symmetric spaces}
Classical references for this section are \cite[Chapter 8, \S 8.B, 8.D, 8.E, 8,G]{thirion2007sous}, \cite[Chapter III, \S 1--4]{guivarc2012compactifications} and  \cite[Chapter IV, Chapter V, Chapter VI]{helgason1978differential}. 

Let $K$ be a maximal compact subgroup of $G$.
Then $X=G/K$ is a globally symmetric space of non-compact type.
The group $G$ is the identity component of its isometry group.
It acts transitively on $X$, by left multiplication. 
We fix a point $o=K\in X$.
Then $K$ is in the fixed point set of the involutive automorphism induced by the geodesic symmetry in $o$ (cf. \cite[Chapter VI, Thm 1.1]{helgason1978differential}).

Denote by $\mathfrak{g}$ (resp. $\mathfrak{k}$) the Lie algebra of $G$ (resp. $K$).  
The differential of the involutive automorphism induced by the geodesic symmetry in $o$ is a \emph{Cartan involution} of $\mathfrak{g}$. 
Then $\mathfrak{k}$ is the eigenspace of the eigenvalue $1$ (for the Cartan involution) and we denote by $\mathfrak{p}$ the eigenspace of the eigenvalue $-1$.   
The decomposition $\mathfrak{g}=\mathfrak{k}\oplus \mathfrak{p}$ is a \emph{Cartan decomposition}.

\subsection{Flats, Weyl Chambers, classical decompositions}
A \emph{flat} of the symmetric space $X$ is a totally geodesic, isometric embedding of a Euclidean space.
We are interested in flats of maximal dimension in $X$, called \emph{maximal flats}.
One can construct the space of maximal flats following \cite[Chapter 8, \S 8.D, 8,D]{thirion2007sous} thanks to \cite[Chapter V, Prop. 6.1]{helgason1978differential}. 
Let $\fa\subset \mathfrak{p}$ be a \emph{Cartan subspace} of $\mathfrak{g}$ i.e. a maximal abelian subspace such that the adjoint endomorphism of every element is semisimple. 
We denote by $A$ the subgroup $\exp(\fa)$.  
The \emph{real rank} of the symmetric space $X$, denoted by $r_G$, is the dimension of the real vector space $\fa$.
	\begin{definition}\label{defin_parametrized_flat}
	A \emph{parametrized flats} is an embedding of $\fa$ of the form $gf_0$, where $g\in G$ and $f_0$ is the map defined by  
		\begin{align*}
		f_0: \fa & \longrightarrow X \\
		 v & \longmapsto \exp(v) o \quad .
		\end{align*}
	We denote by $\mathcal{W}(X)$ the set of parametrized flats of $X$. 
	\end{definition}
By definition, the set of parametrized flats is the orbit of $f_0$ under the left-action by multiplication of $G$.
The stabilizer of $f_0$ is the centralizer of $A$ in $K$, denoted by $M$.    
We deduce that the set of parametrized flats $\mathcal{W}(X)$ identifies with the homogeneous space $G/M$.
For any parametrized flat $f\in \cW(X)$, there is an element $g_f$ in $G$ such that $f=g_f f_0$.
Hence, the map 
\begin{align*}\label{flats_weylchambers}
\cW (X) & \overset{\sim}{\longrightarrow} G/M \\
 f  & \longmapsto g_fM 
\end{align*}
is a $G-$equivariant homeomorphism.

For any linear form $\alpha$ on $\fa$, set $\fg_\alpha:= \{v\in \fg \vert \forall u\in \fa , \; [u,v]=\alpha(u)v \}$.
The set of restricted roots is $\Sigma:=\{\alpha\in \fa^*\setminus \{0\} \vert \fg_\alpha\neq \{0\} \}.$ 
The kernel of each restricted root is a hyperplane of $\fa$. 
The \emph{Weyl Chambers} of $\fa$ are the connected components of $\fa \setminus \cup_{\alpha \in\Sigma} \ker(\alpha)$. 
We fix such a component, call it the \emph{positive Weyl chamber} and denote it (resp. its closure) by $\fa^{++}$ (resp. $\fa^+$).

We denote by $N_K(A)$ the normalizer of $A$ in $K$.
The group $N_K(A)/M$ is called the \emph{Weyl group}.
The positive Weyl chamber of $\fa$ allows us to tesselate the maximal flats in the symmetric space $X$.
Indeed, $f_0(\fa^+)$ is a fundamental domain for the action of the Weyl group on the maximal flat $f_0(\fa)$ and $G$ acts transitively on the space of parametrized flats.
Finally, the orbit $G.f_0(\fa^+)$ identifies with the space of parametrized flats, the image of $g.f_0(\fa^+)$ is a \emph{geometric Weyl chamber}.
This explains why the set of parametrized flats is also called the \emph{space of Weyl chambers}.
For any geometric Weyl chamber $f(\fa^+)\in G.f_0(\fa^+)$, the image of $0\in \fa^+$ is the \emph{origin}.
 Furthermore,
$$G/M\simeq \mathcal{W}(X)\simeq G.f_0(\fa^+).$$ 
	\begin{definition}\label{definition_weyl_chamber_flow}
	The right-action of $\fa$ on $\mathcal{W}(X)$ is defined by 
	$\alpha \cdot f: v \mapsto f(v+\alpha)$ for all $\alpha \in \fa$ and $f\in \mathcal{W}(X)$. The \emph{Weyl Chamber Flow}, is
	defined for all $\theta\in \fa_1^{++}$ and $f\in \mathcal{W}(X)$ by
		\begin{align*}
		\phi^\theta(f): \R & \longrightarrow \mathcal{W}(X) \\
		t & \longmapsto \phi_t^\theta(f)=f(v+\theta t)=f(v)e^{\theta t}. 
		\end{align*}	
	\end{definition}
Remark that the Weyl Chamber Flow $\phi^\theta_t$ is also the right-action of the one-parameter subgroup $\exp(t \theta)$ on the space of Weyl chambers.

The set of \emph{positive roots}, denoted by $\Sigma^+$, is the subset of roots which take positive values in the positive Weyl chamber.
The positive Weyl chamber also allows to define two particular nilpotent subalgebras $\fn =\oplus_{\alpha\in \Sigma^+} \fg_\alpha$ and $\fn^- =\oplus_{\alpha\in \Sigma^+} \fg_{-\alpha}$.
Finally, set $A^+:= \exp(\fa^+)$, $A^{++}:=\exp(\fa^{++})$, $N:=\exp(\fn)$ and $N^-:=\exp(\fn^-)$.
For all $a\in A^{++}$, $h_+\in N$, $h_-\in N^-$ notice that 
\begin{equation}\label{equ_unipotents}
a^{-n}h_{\pm}a^n  \underset{\pm \infty}{\longrightarrow} id_G .
\end{equation}

\begin{definition}
	For any $g\in G$, we define, by Cartan decomposition, a unique element $\mu(g)\in \fa^+$ such that $g\in K \exp(\mu(g)) K$. 
	The map $\mu : G \rightarrow \fa^+$ is called the \emph{Cartan projection}. 
\end{definition}
	
	The Cartan projection allows to define an $\fa^+-$valued function on $X\times X$, denoted by $d_{\fa^+}$, following \cite[ Def-Thm 8.38]{thirion2007sous}.
	For any $x,x'\in X$, there exists $g,g'\in G$ so that $x=gK$ and $x'=g'K$, we set
	$$d_{\fa^+}(x,x'):=\mu (g'^{-1}g).$$
	 This function is independent of the choice of $g$ and $g'$. 
Recall \cite[Chapter V, Lemma 5.4]{helgason1978differential} that $\fa$ is endowed with a scalar product coming from the Killing form on $\fg$, and the norm of $d_{\fa^+} (x,x')$ coincides with the distance between $x$ and $x'$ in the symmetric space $X$. \\ 
	
An element of $G$ is \emph{unipotent} if all its eigenvalues are equal to $1$ and equivalently if it is the exponential of a nilpotent element. 
	An element of $G$ is \emph{semisimple} if it is diagonalizable over $\C$, \emph{elliptic} (resp. \emph{hyperbolic}) if it is semisimple with eigenvalues of modulus $1$ (resp. real eigenvalues).
	Equivalently, elliptic (resp. hyperbolic, unipotent) elements are conjugated to elements in $K$ (resp. $A$, $N$). 
	
	Any element $g\in G$ admits a unique decomposition (in $G$) $g=g_eg_hg_u$, called the \emph{Jordan decomposition},
	where $g_e$, $g_h$ and $g_u$ commute and where $g_e$ (resp. $g_h$, $g_u$) is elliptic (resp. hyperbolic, unipotent).  
	The element $g_e$ (resp. $g_h$, $g_u$) is called the \emph{elliptic part} (resp. \emph{hyperbolic part}, \emph{unipotent part}) of $g$.  \\
	
\begin{definition}
	For any element $g\in G$, there is a unique element $\lambda(g)\in \fa^+$ such that the hyperbolic part of $g$ is conjugated to $\exp(\lambda(g))\in A^+$.
	The map $\lambda : G \rightarrow \fa^+$ is called the \emph{Jordan projection}. 
\end{definition}	
	An element $g\in G$ is \emph{loxodromic}  if $\lambda(g) \in \fa^{++}$.
	Since any element of $N$ that commute with $\fa^{++}$ is trivial, the unipotent part of loxodromic elements is trivial.
	Furthermore, the only elements of $K$ that commute with $\fa^{++}$ are in $M$. 
	We deduce that the elliptic part of loxodromic elements are conjugated to elements in $M$.
	Hence, for any loxodromic element $g\in G$, there exists $h_g\in G$ and $m(g)\in M$ so that we can write $g=h_g m(g)e^{\lambda(g)} h_g^{-1}$.
	For any $m\in M$ we can also write $g=(h_g m) (m^{-1} m(g)m ) e^{\lambda(g)} (h_gm)^{-1}$.	 
	This allows us to associate to any loxodromic element $g\in G$, an \emph{angular} part $m(g)$ which is defined up to conjugacy by $M$.
	
	The spectral radius formula \cite[Corollary 5.34]{BenoistQuint}
	$$\lambda(g) =\lim_{n\tv \infty} \frac{1}{n}\mu(g^n)$$
	allows to compute the Jordan projection thanks to the Cartan projection. 

\begin{definition}
	For any $g\in G$, there exists a unique triple $ (k,v,n)\in K\times \fa \times N$ such that $g=k\exp(v)n$.  
	Furthermore, the map 
\begin{align*}
K \times \fa \times N &\longrightarrow G \\
(k,v,n) & \longmapsto ke^v n
\end{align*}	
is a diffeomorphism called the \emph{Iwasawa decomposition}.
\end{definition}

\subsection{Asymptotic Weyl chambers, Busemann-Iwasawa cocycle}

The main references for this subsection are \cite[Chapter 8, \S 8.D]{thirion2007sous}, \cite{guivarc2012compactifications} and \cite{BenoistQuint}.

We endow the space of geometric Weyl chambers with the equivalence relation 
$$ f_1(\fa^+)\sim f_2(\fa^+) \Leftrightarrow \sup_{u\in \fa^{++}} d(f_1(u),f_2(u))<\infty.$$ 
Equivalently, $f_1(\fa^+) \sim f_2(\fa^+)$ if and only if for any $v\in \fa^{++}$, the geodesics $t\mapsto f_1(tv)$ and $t \mapsto f_2(t v)$ are at bounded distance when $t\rightarrow + \infty$.
Equivalence classes for this relation are called \emph{asymptotic Weyl chambers}. 
We denote by $\mathcal{F}(X)$ the set of asymptotic Weyl chambers and by $\eta_0$ the asymptotic class of the Weyl chamber $f_0(\fa^+)$. 

\begin{fait}
The set $\mathcal{F}(X)$ identifies with the \emph{Furstenberg boundary} $G/P$ where $P=MAN$.
Furthermore
$$G/P \simeq \mathcal{F}(X) \simeq K/M \simeq K .\eta_0.$$
\end{fait}
	\begin{proof}
	Since $G$ acts transitively on the space of Weyl chambers, it also acts transitively on the set of asymptotic Weyl chambers.
	
	We show that $P$ is the stabilizer of $\eta_0$.
	For any $g\in G$ and $u\in \fa^{++}$, we compute the distance
	  $$d(gf_0(u),f_0(u))= \Vert d_{\fa^+}(gf_0(u),f_0(u))\Vert=\Vert \mu(e^{-u}ge^u)\Vert.$$
	By Bruhat decomposition (see \cite[Chapter IX, Thm 1.4]{helgason1978differential}),
	there exists an element $w$ in the normalizer of $A$ in $K$ and elements $p_1,p_2\in P=MAN$ so that 
	$g=p_1 w p_2.$   
	Then 
	$$e^{-u}ge^u= \big(e^{-u}p_1e^u\big) e^{-u} (w e^u w^{-1}) w \big(e^{-u}p_2e^u\big).$$
	Note that by equation (\ref{equ_unipotents}), the sets $\lbrace e^{-u}p_i e^u \rbrace_{u\in \fa^{++},i=1,2}$ are bounded.
	Hence, the sets $\lbrace e^{-u} g e^u \rbrace_{u\in \fa^{++}}$ and $\lbrace e^{-u} w e^u w^{-1} \rbrace_{u\in \fa^{++}}$ have the same behavior. 
	Remark now that $ e^{-u} w e^u w^{-1}=e^{-u+Ad(w)u}$, which is bounded uniformly in $\fa^{++}$ only when $w\in M$.
	We deduce that $\lbrace e^{-u} g e^u \rbrace_{u\in \fa^{++}}$ is bounded only when $g\in P$.
	Hence the subgroup $P$ is the stabilizer of the asymptotic class $\eta_0$.

	The geometric Weyl chambers whose origin is $o\in X$ are in the orbit $K. f_0(\fa^+)$. 
Any equivalence class in $\mathcal{F}(X)$ admits, by Iwasawa decomposition, a unique representative in $K. f_0(\fa^+)$.
Moreover, $K/M$ identifies with the orbit $K.f_0(\fa^+)$ since $M$ is the stabilizer of $f_0$ in $K$.
	\end{proof}

	For any asymptotic Weyl chamber $\eta \in \mathcal{F}(X)$ and $g\in G$, consider, by Iwasawa decomposition, the unique element $\sigma(g,\eta)\in \fa$, called the \emph{Iwasawa cocycle}, such that if $k_\eta \in K$ satisfies $\eta=k_\eta \eta_0$, then 
	$$gk_\eta\in K \exp(\sigma(g,\eta))N.$$
The \emph{cocycle relation} holds (cf \cite[Lemma 5.29]{BenoistQuint}) i.e. for all $g_1,g_2 \in G$ and $\eta \in \mathcal{F}(X)$ then
$$\sigma(g_1g_2,\eta)=\sigma(g_1,g_2 \eta)+ \sigma(g_2,\eta).$$
	
	For any pair of points $x,y\in X$, any asymptotic Weyl chamber $\eta \in \mathcal{F}(X)$ and $u\in \fa^{++}$, we consider a representative $f_\eta(\fa^+)$ of $\eta$ and define the \emph{Busemann cocycle} by
	$$\beta_{f_\eta,u}(x,y) = \lim_{t\tv +\infty} d_{\fa^+} (f_\eta(tu),x) - d_{\fa^+} (f_{\eta}(tu),y).$$ 

It turns out that the Busemann cocycle depends neither on the choice of the geometric Weyl chamber in the class $\eta$, nor on the choice of $u\in \fa^{++}$.
We will write $\beta_{f_\eta,u}(x,y)=\beta_{\eta}(x,y)$.
By \cite[Corollary 5.34]{BenoistQuint}, the Iwasawa and Busemann cocycle coincide in the sense that for all $g\in G$, $\eta \in \mathcal{F}(X)$ and $u\in \fa^{++}$,
\begin{equation} \label{eq_Busemann_Iwasawa}
\beta_{f_\eta,u}(g^{-1}o,o)=\sigma(g,\eta). 	
\end{equation}

We associate attractive and repulsive asymptotic geometric Weyl chambers to loxodromic elements of $G$ as follows.

Recall that for any loxodromic element $g\in G$, there is an element $h_g\in G$ and an angular part $m(g)\in M$ so that 
$g=h_g e^{\lambda(g)} m(g)h_g^{-1}.$
We set $g^+:=[h_g.f_0(\fa^+)]$ and $g^-:=[h_g.f_0(-\fa^+)]$. 
Then $g^+\in \ccF(X)$ (resp. $g^-$) is called the \emph{attractive} (resp. \emph{repulsive}) asymptotic Weyl chamber. 
\begin{fait} \label{fait_lox}
	For any loxodromic element $g\in G$, we have $\lambda(g)=\sigma(g, g^+)$. 
	\end{fait}
\begin{proof}
Let $g\in G$ be a loxodromic element.
Consider an element $h_g\in G$ and an angular part $m(g)\in M$ so that $g=h_g e^{\lambda(g)}m(g)h_g^{-1}$.
Denote by $f_g$ the parametrized flat $f_g: v \mapsto h_g e^v o$. 
Then the geometric Weyl chamber $f_g(\fa^+)$ (resp. $f_g(-\fa^+)$) is a representative of the limit points $g^+$ (resp. $g^-$).

Fix any $u\in \fa^{++}$.
Then by equation (\ref{eq_Busemann_Iwasawa}), we deduce 
\begin{align*}
\sigma(g,g^+)=\beta_{f_g,u}(g^{-1}o,o) &= \lim_{t\tv +\infty} d_{\fa^+} (f_g(tu),g^{-1}o) - d_{\fa^+} (f_g(tu),o) \\
&= \lim_{t\tv +\infty} \mu(g h_g e^{tu})-\mu(h_ge^{tu}) \\
&= \lim_{t\tv +\infty} \mu (h_g e^{\lambda(g)+tu}m(g))-\mu(h_g e^{tu}).
\end{align*} 
By left and right $K-$invariance of the Cartan projection, we deduce that 
$\mu (h_g e^{\lambda(g)+tu}m(g))=\mu (h_g e^{\lambda(g)+tu}).$
Hence 
$$\sigma(g,g^+)= \lim_{t\tv +\infty} \mu (h_g e^{\lambda(g)+tu})-\mu(h_g e^{tu}).$$
By Iwasawa decomposition on $h_g$, there exists a unique unipotent element $n\in N$ so that
$h_g \in K e^{\sigma(h_g , \eta_0)}n$.
Hence, for all $t\in \R_+$,
\begin{align*}
\mu (h_g e^{\lambda(g)+tu})-\mu(h_g e^{tu})
&= \mu (e^{\sigma(h_g , \eta_0)}n e^{\lambda(g)+tu} ) \\
& \quad \quad \quad \quad - \mu (e^{\sigma(h_g , \eta_0)}n  e^{tu}) \\
&= \mu (e^{\sigma(h_g , \eta_0) + \lambda(g)+tu}  e^{-\lambda(g)-tu}n e^{\lambda(g)+tu} ) \\
& \quad \quad \quad \quad - \mu (e^{\sigma(h_g , \eta_0)+tu} e^{-tu} n  e^{tu}).
\end{align*}
Since $u \in \fa^{++}$, then for any $t\in \R_+$ large enough, 
$\sigma(h_g , \eta_0) + \lambda(g)+tu$ and $\sigma(h_g , \eta_0)+tu$ are in $\fa^+$.
Furthermore, by equation (\ref{equ_unipotents})~, we deduce 
$$ \underset{t \rightarrow + \infty}{\lim} e^{-\lambda(g)-tu}n e^{\lambda(g)+tu} = \underset{t \rightarrow + \infty}{\lim} e^{-tu}n e^{tu}= id_G.$$
Hence, by continuity of the Cartan projection, when $t\rightarrow + \infty$,
\begin{align*}
\mu (e^{\sigma(h_g , \eta_0) + \lambda(g)+tu}  e^{-\lambda(g)-tu}n e^{\lambda(g)+tu} ) &=  \mu(e^{\sigma(h_g , \eta_0) + \lambda(g)+tu}) +o(1)\\
\mu (e^{\sigma(h_g , \eta_0)+tu} e^{-tu} n  e^{tu}) 
&= \mu (e^{\sigma(h_g , \eta_0)+tu})+o(1),
\end{align*}
and,
\begin{align*}
d_{\fa^+} (f_g(tu),g^{-1}o) - d_{\fa^+} (f_g(tu),o)
&= \sigma(h_g , \eta_0) + \lambda(g)+tu - (\sigma(h_g , \eta_0)+tu ) + o(1) \\
&= \lambda(g) + o(1).
\end{align*}

Finally, $\lambda(g)=\sigma(g, g^+)$.
\end{proof}

\subsection{Hopf parametrization}
Our main reference for this subsection is \cite[Chapter 8, \S 8.G.2]{thirion2007sous}.

In the geometric compactification of symmetric spaces of non-compact type, any bi-infinite geodesic defines opposite points in the geometric boundary.
In a similar way, we introduce asymptotic Weyl chambers in general position.

We endow the product $\mathcal{F}(X)\times \mathcal{F}(X)$ with the diagonal left $G-$action.
For any $(\xi, \eta) \in \mathcal{F}(X)\times \mathcal{F}(X)$ and $g\in G$, we set 
$g.(\xi, \eta):= (g.\xi, g.\eta).$
For any parametrized flat $f\in \mathcal{W}(X)$, denote by $f_+$ (resp. $f_-$) the asymptotic class of the geometric Weyl chamber $f(\fa^+)$ (resp. $f(-\fa^+)$). 
Then the following map 
	\begin{align*}
\mathcal{H}^{(2)}	:\cW(X) & \longrightarrow  \mathcal{F}(X)\times \mathcal{F}(X)  \\
	f & \longmapsto   (f_+, f_- )
	\end{align*}
is $G-$equivariant and continuous.

Two asymptotic Weyl chambers $\xi, \eta \in \mathcal{F}(X)$ are in \emph{general position} or \emph{opposite}, if they are in the image $\mathcal{H}^{(2)}(\cW(X))$ i.e. if there exists a parametrized flat $f\in \mathcal{W}(X)$ such that
the geometric Weyl chamber $f_+$ (resp. $f_-$) is a representative of $\xi$ (resp. $\eta$).

We denote by $\mathcal{F}^{(2)}(X)$ the set of asymptotic Weyl chambers in general position. 
The product topology on the product space $\mathcal{F}(X) \times \mathcal{F}(X)$ induces a natural topology on $\mathcal{F}^{(2)}(X)$.

\begin{fait}[\S 3.2 \cite{thirion2009proprietes}]
The set $\mathcal{F}^{(2)}(X)$ identifies with the homogeneous space $G/AM$.
Furthermore, if we denote by $\eta_0$ (resp. $\check{\eta}_0$) the asymptotic class of the Weyl chamber $f_0(\fa^+)$ (resp. $f_0(-\fa^+)$), then
$$ G.(\eta_0,\check{\eta}_0) \simeq \mathcal{F}^{(2)}(X)\simeq G/AM .$$
\end{fait}
%

The \emph{Hopf coordinates map} is defined by
	$$ \cH:\begin{array}{cl}
	\cW(X) & \longrightarrow  \ccF^{(2)}(X)  \times \fa \\
	f & \longmapsto   \Big(f_+, f_-;  \beta_{f_+}(f(0),o)\Big).
	\end{array}$$

We define the left $G-$action on the skew product $\ccF^{(2)}(X)  \times \fa$ as follows.
For any $g\in G$ and $(\xi,\eta ; v) \in \ccF^{(2)}(X)  \times \fa$, we set
$$g.(\xi,\eta ; v)=(g.\xi,g.\eta ; v+ \beta_{g.\xi}(g.o,o) ).$$

The right $\fa-$action defined for any $\alpha\in \fa$ and $(\xi,\eta ; v) \in \ccF^{(2)}(X)  \times \fa$ by 
$$\alpha \cdot (\xi,\eta ; v)= (\xi,\eta ; v+\alpha)$$
is called the right $\fa-$action by \emph{translation}.

Similarly, for any $\theta\in \fa_1^{++}$, we define the Weyl chamber flow $\phi^\theta$ on the skew product, for all $(\xi,\eta ; v) \in \ccF^{(2)}(X)  \times \fa$ and $t\in \R_+$,
$$\phi_t^\theta (\xi,\eta ; v)= (\xi, \eta; v+ \theta t) .$$ 

\begin{prop}[Proposition 8.54 \cite{thirion2007sous}]\label{coord_hopf}
The Hopf coordinates map is a $(G,\fa)-$equivariant homeomorphism in the sense that: 
\begin{itemize}
\item[(i)] The left-action of $G$ on $\mathcal{W}(X)$ identifies, via the Hopf coordinates map, with the left $G-$action on the skew product $ \ccF^{(2)}(X)  \times \fa$; 
\item[(ii)] The right-action of $\fa$ on $\mathcal{W}(X)$ identifies, via the Hopf coordinates map, with the right $\fa-$action by translation on the skew product $ \ccF^{(2)}(X)  \times \fa$.  
\end{itemize}
Furthermore, for any $\theta\in \fa_1^{++}$ and $t\in \R_+$, for all $f\in \mathcal{W}(X)$, we obtain 
$$\mathcal{H}(\phi_t^\theta(f))=\phi_t^\theta(\mathcal{H}(f)).$$
\end{prop}

\section{Loxodromic elements}
We first study loxodromic elements in $GL(V)$ for $V$  a real vector space of finite dimension endowed with a Euclidean norm $\Vert .\Vert$.
Then we give some background on representations of semisimple Lie groups.
Finally, we study the dynamical properties of the representations of $G$ acting on the projective space of those representations.
\subsection{Proximal elements of $\mathrm{GL}(V)$}
Denote by $X=\mathbb{P}(V)$ the projective space of $V$.
We endow $X$ with the distance 
$$d(\R x,\R y)=\inf \lbrace \Vert v_x-v_y \Vert \;  \vert  \; \Vert v_x\Vert=\Vert v_y \Vert=1 ,\;  v_x\in \R x , \; v_y\in \R y  \rbrace .$$
For $g\in \mathrm{End}(V)$, denote by $\lambda_1(g)$ its \emph{spectral radius}.

\begin{definition}\label{definition_proximal}
	An element $g\in \mathrm{End}(V)\setminus \lbrace 0\rbrace$ is \emph{proximal} on $X$ if it has a unique eigenvalue $\alpha \in \mathbb{C}$ such that $\vert \alpha\vert=\lambda_1(g)$ and this eigenvalue is simple (therefore $\alpha$ is a real number).
	Denote by $V_+(g)$ the one dimensional eigenspace corresponding to $\alpha$ and $V_-(g)$ the supplementary $g$-invariant hyperplane. 
	In the projective space, denote by $x_+(g)=\mathbb{P}(V_+(g))$ (resp. $X_-(g)=\mathbb{P}(V_-(g))$) the \emph{attractive point} (resp. the \emph{repulsive hyperplane}). 
	\end{definition}

The open ball centered in $x\in X$ of radius $\varepsilon>0$ is denoted by $B(x,\varepsilon)$.  
For every subset $Y\subset X$, we denote by $\mathcal{V}_\varepsilon(Y)$ the open $\varepsilon-$neighbourhood of $Y$.
The following definition gives uniform control over the geometry of proximal elements (parametrized by $r$) and their contracting dynamics (parametrized by $\varepsilon$).

	\begin{definition}\label{defin_epsilon_prox_dim1}
	Let $0<\varepsilon\leq r$. A proximal element $g$ is $(r,\varepsilon)$-proximal if $d(x_+(g),X_-(g))\geq 2r$, $g$ maps $\mathcal{V}_\varepsilon(X_-(g))^c$ into the ball $B(x_+(g),\varepsilon)$ and its restriction to the subset $\mathcal{V}_\varepsilon(X_-(g))^c$ is an $\varepsilon$-Lipchitz map.
	\end{definition}
We give three remarks that follow from the definition.
	\begin{itemize}\label{remarque_prox}
	\item[1)] If an element is $(r,\varepsilon)$-proximal, then it is $(r',\varepsilon)$-proximal for $\varepsilon \leq r'\leq r$,
	\item[2)] If an element is $(r,\varepsilon)$-proximal, then it is $(r,\varepsilon')$-proximal for $r \geq \epsilon'\geq  \varepsilon$.
	\item[3)] If $g$ is is $(r,\varepsilon)$-proximal, then $g^n$ is also is $(r,\varepsilon)$-proximal for $n\geq 1$.
	\end{itemize}

The numbers $r$ and $\varepsilon$ depend on the metric of the projective space, which, in our case, depends on the choice of the norm on the finite dimensional vector space.
However, in \cite[Remark 2.3]{sert2016} Sert claims the following statement.
We provide a proof for completeness. 

	\begin{lemme}\label{fait_rem_Cagri} For every proximal transformation $g$, there exists $r>0$ and $n_0\in \mathbb{N}$ such that for all $n\geq n_0$ large enough, $g^n$ is $(r,\varepsilon_n)$-proximal with $\varepsilon_n \underset{n\rightarrow \infty}{\rightarrow} 0$. 
	\end{lemme}
Since $\mathrm{GL}(V)$ is endowed with a Euclidean norm, it admits a canonical basis $(e_j)_{1 \leq j \leq \dim (V)}$.
We set $x_0:= \mathbb{P}(e_1)$ and $H_0:= \mathbb{P}(\oplus_{j=2}^{\dim(V)} \R e_j)$.
Recall that $\mathrm{GL}(V)$ admits a polar decomposition i.e. for any $g\in \mathrm{GL}(V)$, there exists orthogonal endomorphisms $k_g,l_g\in \mathrm{O}(V)$ and a unique symmetric endomorphism $a_g$ of eigenvalues $(a_g(j))_{1\leq j \leq \dim (V)}$ with $a_g(1)\geq a_g(2) \geq ... \geq a_g(\dim (V))$ such that $g=k_ga_gl_g$.  
Let us introduce a key \cite[Lemma 3.4]{breuillard2003dense}, due to Breuillard and Gelander, which is needed to obtain the Lipschitz properties.	

\begin{lemme}[\cite{breuillard2003dense}]\label{lemme_lipschitz}
Let $r,\delta \in (0,1]$. Let $g\in \mathrm{GL}(V)$.
If $\big\vert\frac{ a_g(2) }{ a_g(1) }\big\vert \leq \delta$, then $g$ is $\frac{\delta}{r^2}-$Lipschitz on $\mathcal{V}_r(l_g^{-1}H_0)^c$.
\end{lemme}

	\begin{proof}[Proof of Lemma \ref{fait_rem_Cagri}]
Let $g\in \mathrm{GL}(V)$ be a proximal element.
Set $r:= \frac{1}{2}d(x_+(g),X_-(g))$. By proximality, $r$ is positive.
Let us prove that for all $0<\varepsilon\leq r$, there exists $n_0$ such that $g^n$ is $(r,\varepsilon)-$proximal for all $n\geq n_0$.

Denote by $\pi_g$ the projector of kernel $V_-(g)$ and of image $V_+(g)$.
Then 
$$\frac{g^n}{\lambda_1(g)^n}=\pi_g +  \frac{ g_{\vert V_-(g)}^n }{ \lambda_1(g)^n }.$$
By proximality, the highest eigenvalue of $g_{\vert V_-(g)}$ is strictly smaller than $\lambda_1(g)$. 
It follows immediately by the Spectral Radius Formula that 
$\frac{g^n}{\lambda_1(g)^n} \underset{n\rightarrow +\infty}{\longrightarrow} \pi_g.$
Hence for any $y\in X\setminus  X_-(g)$, uniformly on any compact subset of $X\setminus  X_-(g)$,
$$g^n.y \underset{n\rightarrow +\infty}{\longrightarrow} x_+(g).$$ 

It remains to show the Lipschitz properties of $g^n$, for $n$ big enough. 
For all $n\in \mathbb{N}$, we denote by $k_n,l_n$ (resp. $a_n$) the orthogonal (resp. symmetric) components of $g^n$ so that $g_n=k_n a_n l_n$. 
We also set $x_n:=k_n x_0$ and $H_n:= l_n^{-1}H_0$.

For any $n\geq 1$, denote by $p_{x_n,H_n}$ the endomorphism of norm 1 such that $\mathbb{P}(\mathrm{im} (p_{x_n,H_n}))=x_n$ and $\mathbb{P}(\ker(p_{x_n,H_n}))=H_n$.
Then by polar decomposition, 
$$\frac{g^n}{a_n(1)}=p_{x_n,H_n} + O\bigg( \frac{ a_n(2) }{ a_n(1) }\bigg).$$

By the Spectral Radius Formula, 
$\big\vert \frac{ a_n(2) }{ a_n(1) }\big\vert^{\frac{1}{n}} \underset{n\rightarrow \infty}{\longrightarrow} \frac{\lambda_1(g_{\vert V_-(g)})}{\lambda_1(g)}  <  1.$ 
Hence
$$\underset{n\rightarrow \infty}{\lim} \frac{ a_n(2) }{ a_n(1) } = 0.$$

Let $(x,H)$ be an accumulating point of the sequence $(x_n,Y_n)_{n\geq 1}$.
Then there is a converging subsequence $x_{\varphi(n)},Y_{\varphi(n)}\underset{n\rightarrow +\infty}{\longrightarrow} x,Y$.
Denote by $p_{x,Y}$ the endomorphism of norm 1 such that $\mathbb{P}(\mathrm{im} (p_{x,Y}))=x$ and $\mathbb{P}(\ker(p_{x,Y}))=Y$.
Then,  
$$\frac{g^{\varphi(n)}}{a_{\varphi(n)}(1)} \underset{n\rightarrow +\infty}{\longrightarrow} p_{x,Y}.$$
It allows us to deduce in particular, that for any $y\in X\setminus \lbrace H, X_-(g) \rbrace$,
$$g^{\varphi(n)}.y \underset{n\rightarrow +\infty}{\longrightarrow} x.$$ 
However, by proximality of $g$ and uniqueness of the limit, we obtain that $x=x_+(g)$.  

Similarly, by duality, we obtain that $Y=X_-(g)$. Hence $(x_n,Y_n)_{n\geq 1}$ converges towards $(x_+(g),X_-(g))$. 

Fix $0<\varepsilon\leq r$. 
Then for $n$ large enough, the inclusion $\mathcal{V}_\varepsilon(X_-(g))\supset \mathcal{V}_{\frac{\varepsilon}{2}}(H_n)$ holds.
By Lemma \ref{lemme_lipschitz}, the restriction of $g^n$ to $\mathcal{V}_\varepsilon(X_-(g))^c \subset \mathcal{V}_{\frac{\varepsilon}{2}}(H_n)^c$ is then a $\big\vert \frac{ a_n(2) }{ a_n(1) }\big\vert \frac{4}{\varepsilon^2}-$Lipschitz map.
Finally, for $n$ large enough so that $\big\vert \frac{ a_n(2) }{ a_n(1) }\big\vert \frac{4}{\varepsilon^2}<\varepsilon$, the restriction of $g^n$ to  $\mathcal{V}_\varepsilon(X_-(g))^c$ is $\varepsilon-$Lipschitz.  
	\end{proof}	

The following proximality criterion is due to Tits \cite{Tits1971} and one can find the statement under this form in \cite{benoist2000proprietes}.
	\begin{lemme}\label{critère_de_proximalité_Tits}
	Fix $0<\varepsilon\leq r$. 
	Let $x\in \mathbb{P}(V)$ and a hyperplane $Y\subset \mathbb{P}(V)$ such that $d(x,Y)\geq 6r$.
	Let $g\in \mathrm{GL}(V)$.
	If  
		\begin{itemize}
		\item[(i)] $g\mathcal{V}_{\varepsilon}(Y)^c\subset B(x,\varepsilon)$,
		\item[(ii)] $g$ restricted to $\mathcal{V}_{\varepsilon}(Y)^c$ is $\varepsilon-$Lipschitz,
		\end{itemize}  
	then $g$ is $(2r,2\varepsilon)-$proximal.
	Furthermore, the attractive point $x_+(g)$ is in $B(x,\varepsilon)$ and the repulsive hyperplane $X_-(g)$ in a $\varepsilon-$neighbourhood of $Y$.
	\end{lemme}
\begin{corollaire} \label{corollaire_produit_prox}
Fix $0<\varepsilon\leq r$. 
Let $g\in \mathrm{GL}(V)$ be a $(r,\varepsilon/2)-$proximal element so that $d(x_+(g),X_-(g))\geq 7r$.

Then for any $h\in \mathrm{GL}(V)$ so that $\Vert h-id_V\Vert\leq \varepsilon/2$, the product $gh$ is $(2r,2\varepsilon)-$proximal, with $x_+(gh)\in B(x_+(g),\varepsilon)$.
\end{corollaire}	
\begin{proof}
Consider a $(r,\varepsilon/2)-$proximal element $g$ and $h\in \mathrm{GL}(V)$ as in the hypothesis.

Remark that $gh$ maps $h^{-1}\mathcal{V}_{\varepsilon/2}(X_-(g))^c$ towards the open ball $B(x_+(g),\varepsilon/2)$.
Furthermore, by proximality of $g$, the restriction of $gh$ to $h^{-1}\mathcal{V}_{\varepsilon/2}(X_-(g))^c$ is $\varepsilon/2-$Lipschitz.

Since $h$ is close to $id_V$, then $\mathcal{V}_{\varepsilon}(h^{-1}X_-(g))^c \subset h^{-1}\mathcal{V}_{\varepsilon/2}(X_-(g))^c$. 
Hence $gh$ restricted to $\mathcal{V}_{\varepsilon}(h^{-1}X_-(g))^c$ is $\varepsilon-$Lipschitz of image in the open ball $B(x_+(g),\varepsilon)$. 
Furthermore, $d(x_+(g),h^{-1}X_-(g))\geq d(x_+(g),X_-(g))-\varepsilon > 7r-\varepsilon \geq 6r.$

Finally, by Lemma \ref{critère_de_proximalité_Tits}, we deduce that $gh$ is $(2r,2\varepsilon)-$proximal, with
$x_+(gh)\in B(x_+(g),\varepsilon)$.
\end{proof}
	
For all proximal elements $g, h$ of $\mathrm{End}(V)$ such that $x_+(h)\notin X_-(g)$, 
we consider two unit eigenvectors $v_+(h)\in x_+(h)$ and $v_+(g)\in x_+(g)$ 
and denote by $\alpha(g,h)$ the unique real number such that 
$v_+(h) - \alpha(g,h) v_+(g) \in V_-(g)$. 
A priori, $\alpha(g,h)$ depends on the choice of the unit vectors.
However, its absolute value does not.

Given $g_1,... g_l$ of $End(V)$, set $g_0=g_l$ 
and assume $x_+(g_{i-1})\notin X_-(g_{i})$ for all $1\leq i\leq l$. 
We set 
$$\nu_1(g_l,...,g_1) =\sum_{1\leq j\leq l}\log  |\alpha (g_{j} ,g_{j-1})|.$$
This product does not depends on the choices of unit eigenvectors for $g_j$.

The following proposition explains how to control the spectral radius $\lambda_1(\g)$ when $\g$ is a product of $(r,\epsilon)$-proximal elements.

\begin{prop}[\cite{benoist2000proprietes}]\label{prop - lemme 1.4 de Benoist ii } For all $0<\epsilon \leq r$, there exist positive constants $C_{r,\epsilon}$ such that for all $r>0$,  
	$\lim_{\epsilon\tv 0} C_{r,\epsilon} =0$ and such that the following holds.
	If $\g_1, ... \g_l$ are $(r,\varepsilon)$-proximal elements, such that $d(x_+(\g_{i-1}),X_-(\g_i))\geq 6r$ for all $1\leq i\leq l$ with $\g_0=\g_{l}$,
	then for all $n_1,...,n_l\geq 1$,
	$$ \left| \log \big( \lambda_1(\g_l^{n_l}...\g_1^{n_1}) \big) -\sum_{i=1}^l n_i \log \big(\lambda_1(\g_i)\big)  - \nu_1(\g_l,...,\g_1) \right| \leq lC_{r,\epsilon}.$$
	Furthermore, the map $\g_l^{n_l}...\g_1^{n_1}$ is $(2r,2\varepsilon)-$proximal with 
	$x_+(\g_l^{n_l}...\g_1^{n_1})\in B(x_+(\g_l),\varepsilon)$ and $X_-(\g_l^{n_l}...\g_1^{n_1})\subset \mathcal{V}_{\varepsilon}(X_-(\g_1))$.
	\end{prop}

	\begin{proof}
	Taking the logarithm in Benoist's \cite[Lemma 1.4]{benoist2000proprietes} gives us the first part of the statement (the estimates).
	We only give a proof of the proximality and the localisation of the attractive points and repulsive hyperplane. 
	
	Let $n_1,...,n_l\geq 1$ and assume that $0<\epsilon \leq r$ and $\epsilon<1$.
	Let us prove that $g_n:=\g_l^{n_l}...\g_1^{n_1}$ verifies the assumptions (i) (ii) of the proximality criterion Lemma \ref{critère_de_proximalité_Tits}. 
	More precisely, we prove by induction on $l$ that $g_n$ restricted to $\mathcal{V}_{\varepsilon}(X_-(\g_1))^c$ is $\varepsilon-$Lipschitz and $g_n \mathcal{V}_{\varepsilon}(X_-(\g_1))^c\subset B(x_+(\g_l),\varepsilon)$. 
	
	By $(r,\varepsilon)-$proximality of $\g_1^{n_1}$, the restriction of $\g_1^{n_1}$ to $\mathcal{V}_{\varepsilon}(\g_1)^c$ is $\varepsilon-$Lipschitz and $\g_1^{n_1} \mathcal{V}_{\varepsilon}(\g_1)^c\subset B(x_+(\g_1),\varepsilon)$.
	
	Assume that for some $1\leq i\leq l$ that $\g_i^{n_i}...\g_1^{n_1}$ restricted to $\mathcal{V}_{\varepsilon}(X_-(\g_1))^c$ is $\varepsilon-$Lipschitz and $\g_i^{n_i}...\g_1^{n_1} \mathcal{V}_{\varepsilon}(X_-(\g_1))^c\subset B(x_+(\g_i),\varepsilon)$.
	Since $d(x_+(\g_{i}),X_-(\g_{i+1}))\geq 6r$ and $0<\epsilon \leq r$ we obtain $B(x_+(\g_i),\varepsilon)\subset \mathcal{V}_{\varepsilon}(X_-(\g_{i+1}))^c$. 
	Then using $(r,\varepsilon)-$proximality of $\g_{i+1}$, its restriction to $B(x_+(\g_i),\varepsilon)$ is $\varepsilon-$Lipschitz and $\g_{i+1}^{n_{i+1}}B(x_+(\g_i),\varepsilon) \subset B(x_+(\g_{i+1}),\varepsilon)$. 
	Hence by induction hypothesis and using $\epsilon<1$, the map $\g_{i+1}^{n_{i+1}}...\g_1^{n_1}$ restricted to $\mathcal{V}_{\varepsilon}(X_-(\g_1))^c$ is $\varepsilon-$Lipschitz and $\g_{i+1}^{n_{i+1}}...\g_1^{n_1} \mathcal{V}_{\varepsilon}(X_-(\g_1))^c\subset B(x_+(\g_{i+1}),\varepsilon)$. 
	
	We conclude the proof.
	By assumption, $d(x_+(\g_{l}),X_-(\g_1))\geq 6r$. Finally, by Lemma \ref{critère_de_proximalité_Tits} we deduce $(2r,2\varepsilon)-$proximality of $g_n$ with $x_+(g_n)\in B(x_+(\g_l),\varepsilon)$ and $X_-(g_n)\subset \mathcal{V}_{\varepsilon}(X_-(\g_1))$.
	\end{proof}
	
The previous proposition motivates the next definition.
	\begin{definition}\label{defin_epsilon_schottky_dim1}
	Let $0<\varepsilon\leq r$. A semigroup $\G \subset \mathrm{GL}(V)$ is \emph{strongly $(r,\varepsilon)$-Schottky} if
	\begin{itemize}
	\item[(i)] every $h\in \G$ is $(r,\varepsilon)-$proximal, 
	\item[(ii)] $d(x_+(h),X_-(h'))\geq 6r$ for all $h,h'\in \G$. 
	\end{itemize} 
	We also write that $\G$ is a \emph{strong $(r,\varepsilon)$-Schottky semigroup}.
	\end{definition}

\subsection{Representations of a semisimple Lie group $G$}

Let $(V,\rho)$ be a representation of $G$ in a real vector space of finite dimension. 
For every character $\chi$ of $\fa$, denote the associated eigenspace by $V_\chi:= \lbrace v\in V \; \vert \; \forall a\in \fa, \; \rho(a)v=\chi(a)v   \rbrace$.
The set of \emph{restricted weights} of $V$ is the set $\Sigma(\rho):=\lbrace \chi \vert V_\chi \neq 0 \rbrace$.
Simultaneous diagonalization leads to the decomposition $V=\underset{\chi \in \Sigma(\rho)}{\oplus}V_\chi $. 
The set of weights is partially ordered as follows
$$\big( \chi_1\leq \chi_2 \big) \Leftrightarrow \big( \forall a\in A^+, \quad  \chi_1(a)\leq \chi_2(a) \big).$$
Whenever $\rho$ is irreducible, the set $\Sigma(\rho)$ has a highest element $\chi_{max}$ which is the \emph{highest restricted weight} of $V$.
Denote by $V_\rho$ the eigenspace of the highest restricted weight, and by $Y_\rho$ the $\fa-$invariant supplementary subspace of $V_\rho$ i.e. $Y_\rho:=\ker(V_{\chi_{max}}^*)= \underset{\chi \in \Sigma(\rho)\setminus \lbrace \chi_{max}\rbrace}{\oplus}V_\chi $.

The irreducible representation $\rho$ is \emph{proximal} when $\dim (V_{\chi_{max}})=1$.
The following Lemma can be found in \cite[Lemma 5.32]{BenoistQuint}. It is due to Tits \cite{Tits1971}.

Denote by $\Pi \subset \Sigma^+$ the subset of \emph{simple roots} of the set of positive roots for the adjoint representation of $G$.

\begin{lemme}[\cite{Tits1971}]\label{lemme_tits}
	For every simple root $\alpha \in \Pi$, there exists a proximal irreducible algebraic representation $(\rho_\alpha, V_\alpha)$ of $G$ whose highest weight $\chi_{max,\alpha}$ is orthogonal to $\beta$ for every simple root $\beta \neq \alpha$. 
	
	These weights $(\chi_{max,\alpha})_{\alpha\in \Pi}$ form a basis of the dual space $\fa^*$.
	
	Moreover, the map 
	\begin{align*}
	\ccF(X) &\overset{y}{\longrightarrow} \prod_{\alpha\in \Pi} \mathbb{P}(V_\alpha) \\
	\eta:=k_\eta \eta_0 & \longmapsto \big( y_\alpha(\eta):= \rho_\alpha (k_\eta) V_{\rho_\alpha} \big)_{\alpha\in \Pi}
	\end{align*}
	is an embedding of the set of asymptotic Weyl chambers in this product of projective spaces.
	\end{lemme}	
We also define a dual map $H:\ccF(X) \rightarrow \prod_{\alpha\in \Pi} Gr_{ \dim(V_\alpha)-1}(V_\alpha)$ as follows.
For every $\eta \in \ccF(X)$, let $k_\eta\in K$ be an element so that $\eta=k_\eta \check{\eta}_0$ then
	\begin{align*}
	\ccF(X) &\overset{Y}{\longrightarrow}  \prod_{\alpha\in \Pi} Gr_{ \dim(V_\alpha)-1}(V_\alpha) \\
	\xi:=k_\xi \check{\eta}_0 & \longmapsto \big( Y_\alpha(\xi):= \rho_\alpha (k_\xi) Y_{\rho_\alpha} \big)_{\alpha\in \Pi}.
	\end{align*}

The maps $y$ and $Y$ provide us two ways to embed the space of asymptotic Weyl chambers $\mathcal{F}(X)$.

\begin{corollaire}\label{cor_flags_pos_gen_grassmann}
The map 
	\begin{align*}
	\ccF^{(2)}(X) &\longrightarrow 
		\prod_{\alpha\in \Pi} \mathbb{P}(V_\alpha)\oplus Gr_{\dim(V)-1}(V_\alpha) \\
	(f_+,f_-) & \longmapsto 
		\big(y_\alpha(f_+) \oplus Y_\alpha(f_-) \big)_{\alpha\in \Pi}. 
	\end{align*}
	is a $G-$equivariant embedding of the space of flags in general position into this product of projective spaces in general position i.e. the associated subspaces are in direct sum.	
\end{corollaire}

Now we give an interpretation of the Cartan projection, the Iwasawa cocycle and the Jordan projection in terms of representations of $G$. 
The complete proof can be found in \cite{BenoistQuint}.
	\begin{lemme}[Lemma 5.33 \cite{BenoistQuint}]\label{lemme_repr_cartan_jordan_iwasawa}
	 Let $\alpha\in \Pi$ be a simple root and consider $(V_\alpha,\rho_\alpha)$ the proximal representation of $G$ given by Lemma \ref{lemme_tits}. Then
	\begin{itemize}
	\item[(a)] there exists a $\rho_\alpha( K)-$invariant Euclidean norm on $V_\alpha$ such that, for all $a\in A$, the endomorphism $\rho_\alpha(a)$ is symmetric.
	\item[(b)] for such a norm and the corresponding subordinate norm on $\mathrm{End}(V_\alpha)$, for all $g\in G$, $\eta \in \ccF(X)$ and $v_\eta\in y_\alpha(\eta)$, one has
		\begin{itemize}
		\item[(i)] $\chi_{max,\alpha}\big(\mu(g)\big)= \log\big(\Vert \rho_\alpha(g) \Vert \big) $,
		\item[(ii)] $\chi_{max,\alpha}\big(\lambda(g)\big)= \log \big(\lambda_1(\rho_\alpha(g))\big) $,
		\item[(iii)] $\chi_{max,\alpha}\big(\sigma(g,\eta)\big) = \log \frac{ \Vert \rho_\alpha(g)v_\eta \Vert }{\Vert v_\eta \Vert}$.
		\end{itemize}
	\end{itemize}
	\end{lemme}

The following lemma gives estimations on the Cartan projection of products of any pair of elements in $G$.

	\begin{lemme}\label{lemme_cartan_produit}
	There exists a continuous, left and right $K-$invariant, function $h\in G\mapsto C_h\in \R_+$ such that
	\begin{itemize}
	\item[(i)] for any $g\in G$, the Cartan projections $ \mu(gh)- \mu(g)$ and $\mu(hg)-\mu(g)$ are in the ball $ \overline{B_\fa (0,C_h)}$,
	\item[(ii)] for any $\eta \in \mathcal{F}(X)$, the Iwasawa cocycle $\sigma(h,\eta)\in \overline{B_\fa (0,C_h)}$.
\end{itemize}	 
	\end{lemme}
\begin{proof}
Abusing terminology, we say that a function is $K-$invariant when it is $K-$invariant for both  left and right action. 

Let us prove the first point.
For any $\alpha \in \Pi$, we consider the proximal irreducible representation $(\rho_\alpha, V_\alpha)$ of $G$ given by Lemma \ref{lemme_tits}. 

Using Lemma \ref{lemme_repr_cartan_jordan_iwasawa}, we endow each vector space $V_\alpha$ with  $\rho_\alpha (K)-$invariant Euclidean norm.
Classical properties of the norm lead, for all $\alpha\in \Pi$ and every $g,h\in G$, to
\begin{align*}\label{lemme_cartan_produit_ineq1}
\frac{\Vert \rho_\alpha (g) \Vert }{\Vert \rho_\alpha(h^{-1}) \Vert} 
&\leq \Vert \rho_\alpha(gh) \Vert \leq \Vert \rho_\alpha(g) \Vert \Vert \rho_\alpha(h) \Vert, \\
 \frac{1 }{\Vert \rho_\alpha(h^{-1}) \Vert} 
 &\leq \frac{\Vert \rho_\alpha(gh) \Vert}{\Vert \rho_\alpha(g) \Vert}
 \leq  \Vert \rho_\alpha(h) \Vert.
\end{align*}
Note that we obtain the same inequalities for $hg$.s
By Lemma \ref{lemme_repr_cartan_jordan_iwasawa}, we deduce 
\begin{equation}\label{lemme_cartan_produit_ineq2}
-\chi_{max,\alpha}\big(\mu(h^{-1})\big) 
\leq \chi_{max,\alpha}\big(\mu(gh)-\mu(g)\big)
\leq \chi_{max,\alpha}\big(\mu(h)\big). 
\end{equation}
For any $\alpha\in \Pi$, set $h_\alpha:= \max \Big(\chi_{max,\alpha}\big(\mu(h)\big), \chi_{max,\alpha}\big(\mu(h^{-1})\big) \Big)$.
Furthermore, by Lemma \ref{lemme_tits}, the weights $(\chi_{max,\alpha})_{\alpha\in \Pi}$ form a basis of the dual space $\fa^*$.
In other word, they admit a dual basis in $\fa$.
Denote by $C_h>0$ the real number such that $\overline{B_\fa(0,C_h)}$ is the smallest closed ball containing any point of dual coordinates in $\big([-h_\alpha, h_\alpha]\big)_{\alpha \in \Pi}$ for the dual basis of $(\chi_{max,\alpha})_{\alpha\in \Pi}$.
Hence $\overline{B_\fa(0,C_h)} $ is compact and contains $\mu(gh)-\mu(g)$ and $\mu(hg)-\mu(g)$. 

It remains to show that the function $h\mapsto C_h$ is continuous and $K-$invariant.
It is due to the fact that the Cartan projection and the map $h\mapsto \mu(h^{-1})$  are both continuous and $K-$invariant. 
Hence, by taking the supremum in each coordinate, the map $h \mapsto (h_\alpha)_{\alpha \in \Pi}$ is continuous and $K-$invariant.
Furthermore, by definition of $C_h$, we obtain $K-$invariance and continuity of $h\mapsto C_h$.

Similarly, the second point is a direct consequence of Lemma \ref{lemme_repr_cartan_jordan_iwasawa}, (i) and (iii) and of the inequality 
\begin{equation}
\frac{1}{\Vert \rho_\alpha(h^{-1}) \Vert} \leq \frac{\Vert \rho_\alpha(h)(v_\eta) \Vert}{\Vert v_\eta \Vert} \leq \Vert \rho_\alpha(h) \Vert
\end{equation}
where $\eta \in \mathcal{F}(X)$ and $v_\eta\in V_\alpha$ is the associated non trivial vector.

\end{proof}

\subsection{Loxodromic elements}
Let us now study the dynamical properties of the loxodromic elements in the representations of the previous paragraph. \cite[Lemma 5.37]{BenoistQuint} states that any element of $G$ is loxodromic if and only if its image is proximal for every representations given by Lemma \ref{lemme_tits}.	
This allows to extend the notions and results on proximal elements to loxodromic elements in $G$.
\begin{definition}\label{defin_lox_schottky}
An element $g\in G$ is \emph{loxodromic} if its Jordan projection $\lambda(g)$ is in the interior of the Weyl chamber $\mathfrak{a}^{++}$ or (equivalently) if for all $\alpha \in \Pi$ the endomorphisms $\rho_\alpha(g)$ are proximal.

Let $0<\varepsilon\leq r$.
An element $g\in G$ is \emph{$(r,\varepsilon)$-loxodromic} if for all $\alpha \in \Pi$ the endomorphisms $\rho_\alpha(g)$ are $(r,\varepsilon)$-proximal.

A semigroup $\G$ of $G$ is \emph{strongly $(r,\varepsilon)$-Schottky} if for all $\alpha \in \Pi$ the semigroups $\rho_\alpha(\G)\subset \mathrm{End}(V_\alpha)$ are strongly $(r,\varepsilon)$-Schottky.
\end{definition}

Attractive and repulsive asymptotic Weyl chambers of loxodromic elements were defined in section  2.2 as follows.	
For any loxodromic element $g\in G$, then $(g^+,g^-):=h_g(\eta_0,\check{\eta}_0) \in \ccF^{(2)}(X)$ where $h_g\in G$ is an element so that there is an angular part $m(g)\in M$ with $g=h_g e^{\lambda(g)}m(g)h_g^{-1}$. 

The $G-$equivariant map $(f_+,f_-)\in \ccF^{(2)}(X) \rightarrow 
		\big( y_\alpha(f_+) \oplus Y_\alpha(g_-) \big)_{\alpha\in \Pi}$ 
given by Corollary \ref{cor_flags_pos_gen_grassmann} allows to caracterize attractive points and repulsive points in $\ccF(X)$ for loxodromic elements.

	\begin{lemme}\label{lemme_carac_lox}
For any loxodromic element $g\in G$, the following statements are true.

\begin{itemize}
\item[(i)] $g^{-1}$ is loxodromic, of attractive point $g^-$ and repulsive point $g^+$,
\item[(ii)]  the image of $(g^+,g^-)\in \ccF^{(2)}(X)$ by the above map is the family of attractive points and repulsive hyperplanes in general position $\big(x_+(\rho_\alpha(g)) \oplus X_-(\rho_\alpha(g)) \big)_{\alpha\in \Pi}$,
\item[(iii)] $g$ contracts any point $\eta \in \ccF(X)$ in general position with $g^-$, to $g^+$ i.e.
$\underset{n \rightarrow + \infty}{\lim} g^n \eta = g^+$,
\item[(iv)] for any nonempty open set $O_-\subset \ccF(X)$ in general position with $g^+$, for any nonempty open neighbourhood  $U_- \subset \ccF(X)$ of $g^-$, there exists $N\in \mathbb{N}$ so that for any $n\geq N$, then
$ O_- \cap g^n U_- \neq \emptyset .$ 
\end{itemize}
	\end{lemme}
	
\begin{proof}
Let $g\in G$ be a loxodromic element, consider an element $h_g\in G$ and an angular part $m(g)\in M$ so that $g=h_g m(g)e^{\lambda(g)}h_g^{-1}$.
Then $g^{-1}=h_g  m(g)^{-1} e^{-\lambda(g)}h_g^{-1}$.
Remark that $-\lambda(g)$ is in the interior of the Weyl chamber $-\fa^+$.
Consider the element of the Weyl group $N_K(A)/M$ whose adjoint action on $\fa$ sends $\fa^+$ onto $-\fa^+$. 
Denote one representative by $k_\iota\in N_K(A)$.
Then $-Ad(k_\iota)(\lambda(g)) \in \fa^{++}$, hence
$$g^{-1}= h_gk_\iota  (k_\iota^{-1} m(g) k_\iota)^{-1} e^{-Ad(k_\iota)(\lambda(g)) }  (h_g k_\iota)^{-1}.$$
Next, we remark that $k_\iota^{-1} M k_\iota$ is in the centralizer of $k_\iota^{-1} A k_\iota=A$, hence $k_\iota^{-1} m(g) k_\iota\in k_\iota^{-1} M k_\iota=M$.
We deduce that $\lambda(g^{-1})=-Ad(k_\iota)(\lambda(g))$ and set $h_{g^{-1}}=h_gk_\iota$ with the angular part $m(g^{-1})=(k_\iota^{-1} m(g) k_\iota)^{-1}$.
Then the pair of attractive and repulsive points of $g^{-1}$ in $\ccF(X)$ is $(h_gk_\iota \eta_0, h_g k_\iota \check{\eta}_0)$.
Since $k_\iota \eta_0=\check{\eta}_0$ and $k_\iota \check{\eta}_0=\eta_0$ we obtain the first statement i.e.
 that $g^-$ (resp. $g^+$) is the attractive (resp. repulsive) point of $g^{-1}$.

For the second point, it suffices to prove that for any loxodromic element $g\in G$, for every $\alpha \in \Pi$, the vector space $\rho_\alpha (h_g) V_{\rho_\alpha}=y_\alpha(g^+)$ is the eigenspace associated to the spectral radius of $\rho_\alpha(g)$ and that $\rho_\alpha (h_g) Y_{\rho_\alpha}=Y_\alpha(g^-)$ is the direct sum of the other eigenspaces.

Let $g\in G$ be a loxodromic element and let $ \alpha\in \Pi$.
By Lemma \ref{lemme_repr_cartan_jordan_iwasawa}, the spectral radius of $\rho_\alpha( g)$ is $\exp(\chi_{max,\alpha}(\lambda(g)))$. We deduce that the eigenspace of the highest eigenvalue is
 $\rho_\alpha (h_g) V_{\rho_\alpha}$. 
 Furthermore, by definition of proximality, $x_+(\rho_\alpha(g))=\mathbb{P}(\rho_\alpha (h_g) V_{\rho_\alpha})=y_\alpha(g^+)$.
 
 Remark that the other eigenvalues of $\rho_\alpha (g)$ are given by the other non maximal restricted weights of the representation $(\rho_\alpha, V_\alpha)$. 
 Hence $\rho_\alpha (h_g) Y_{\rho_\alpha}$ is the direct sum of the other eigenspaces of $\rho_\alpha (h_g)$.
 The projective space $\mathbb{P}(\rho_\alpha (h_g) Y_{\rho_\alpha})$ is thus the repulsive hyperplane of $\rho_\alpha(g)$.
Hence the second statement is true. 
 
 For any point $\eta \in \ccF(X)$ in general position with $g^-$ and for any $\alpha\in \Pi$, the point $y_\alpha(\eta)$ is then in general position with the hyperplane $Y_\alpha(g^-)$.
 Hence $\underset{n \rightarrow + \infty}{\lim} \rho_\alpha(g^n) y_\alpha(\eta) = x_+(\rho_\alpha(g))$. 
This gives the third statement.
 
 For the last statement, we apply the third statement to $g^{-1}$.  
It means that, for any nonempty open set $O_-\subset \ccF(X)$ in general position with $g^+$ and for any nonempty open neighbourhood $U_- \subset \ccF(X)$ of $g^-$, there exists $N\in \mathbb{N}$ so that for any $n\geq N$, then
$$\big(g^{-1}\big)^n O_- \cap U_- \neq \emptyset.$$  
Hence, for any $n\geq N$, 
 $$g^n\big( g^{-n} O_- \cap U_-  )\neq \emptyset,$$  
finally, 
  $$ O_- \cap g^n U_-  \neq \emptyset.$$  

\end{proof}

Lemma \ref{fait_rem_Cagri} and Corollary \ref{corollaire_produit_prox} extend to loxodromic elements.
	\begin{lemme}\label{lemme_rem_Cagri_lox}	
	For every loxodromic element $g\in G$, there exists $r>0$ and $n_0\in \mathbb{N}$ such that for all $n\geq n_0$ large enough, $g^n$ is $(r,\varepsilon_n)$-loxodromic with $\varepsilon_n \underset{n\rightarrow \infty}{\rightarrow} 0$. 
	\end{lemme}

	\begin{corollaire} \label{corollaire_produit_lox}
Fix $0<\varepsilon\leq r$. 
Let $g\in G$ be a $(r,\varepsilon/2)-$loxodromic element so that $d(g^+,g^-)\geq 7r$.

Then for any $h\in G$ so that $\Vert h-id_V\Vert\leq \varepsilon/2$, the product $gh$ is $(2r,2\varepsilon)-$loxodromic, with $(gh)^+\in B(g^+,\varepsilon)$.
\end{corollaire}

Likewise, we generalize estimates of Proposition \ref{prop - lemme 1.4 de Benoist ii } to products of loxodromic elements of $G$ in general configuration. 

Given $l$ loxodromic elements $g_1,... g_l$ of $G$, set $g_0=g_l$ 
and assume that the asymptotic points $g_{i-1}^+$ and $g_{i}^-$ are opposite for all $1\leq i\leq l$. 
Thanks to lemma \ref{lemme_tits}, there exists a unique element $\nu=\nu(g_1,...,g_l)\in \fa$ whose coordinates in the dual basis of $(\chi_{\alpha , max})_{\alpha \in \Pi}$ are
$$\big( \chi_{\alpha , max}(\nu) \big)_{\alpha\in \Pi}:=\big( \nu_1(\rho_\alpha(g_1),...,\rho_\alpha(g_l))  \big)_{\alpha\in \Pi}.$$
This product does not depends on the choices of unit eigenvectors for $g_j$.
The product of projective spaces $\prod_{\alpha\in \Pi} \mathbb{P}(V_\alpha)$ is endowed with the natural distance.	
	\begin{prop}[Benoist\cite{benoist2000proprietes}]\label{prop - lemme 3.6 de Benoist ii } For all $0<\epsilon \leq r$, there exist positive constants $C_{r,\epsilon}$ such that for all $r>0$,  
	$\lim_{\epsilon\tv 0} C_{r,\epsilon} =0$ and such that the following holds.
	If $\g_1, ... \g_l$ are $(r,\varepsilon)$-loxodromic elements, such that for all $1\leq i\leq l$ with $\g_0=\g_{l}$ we have $d(y(\g_{i-1}^+),H(\g_i^-))\geq 6r$,
	then for all $n_1,...,n_l\geq 1$
	$$  \lambda(\g_l^{n_l}...\g_1^{n_1}) -\sum_{i=1}^l n_i \lambda(\g_i)  - \nu(\g_l,...,\g_1)  \in B_\fa(0, lC_{r,\epsilon}).$$
	Furthermore, the map $g:=\g_l^{n_l}...\g_1^{n_1}$ is $(2r,2\varepsilon)-$loxodromic with 
	$y(g^+)\in B(y(\g_l^+),\varepsilon)$ and $H(g^-)\in \mathcal{V}_\varepsilon(H(\g_1^-))$.
	\end{prop}
Using Proposition \ref{prop - lemme 3.6 de Benoist ii }, one can construct finitely generated, strong $(r,\varepsilon)$-Schottky semigroups as follows.
Let $0<\varepsilon\leq r$.

Let $S\subset G$ be a family of $(r/2,\varepsilon/2)-$loxodromic elements such that $d(y(h^+),H(h'^-))\geq 7r$ for all $h,h'\in S$.
Denote by $\G'$ the semigroup generated by $S$. 
Then every element $g\in \G$ is a noncommuting product of proximal elements of the form $g_l^{n_l}...g_1^{n_1}$ with $n_1,...,n_l\geq 1$ and $g_i\neq g_{i+1}\in S$ for all $1\leq i<l$.
By Proposition \ref{prop - lemme 3.6 de Benoist ii }, we deduce $d(y(g^+),H(g^-)) \geq d(y(g_l^+),H(g_1^-))-\varepsilon \geq 6r$ and that $g$ is $(r,\varepsilon)-$loxodromic.
Thus, $\G'$ is strongly $(r,\varepsilon)$-Schottky.
\section{Topological transitivity}
Recall the definition of topological transitivity. We denote by $\fa_1^{+}$ (resp. $\fa_1^{++}$ ) the intersection of the unit sphere in $\fa$ with $\fa^+$ (resp. $\fa^{++}$). 
		\begin{definition}\label{definition_topological_transitivity}
		Let $\widetilde{\Omega} \subset \mathcal{W}(X)$ a $\G$-invariant and $\fa$-invariant subset of parametric flats $\mathcal{W}(X)$.
		Let $\Omega:=\G \backslash \widetilde{\Omega}$. 
		Fix a direction $\theta\in \fa_1^{++}$.
		The Weyl chamber flow $\phi_\R^{\theta}$ is \emph{topologically transitive} on $\Omega$ if for all open nonempty subsets $U,V \subset \Omega$, there exists $t_n\rightarrow +\infty$ such that for every $n\geq 1$, we have $U\cap \phi_{t_n}^\theta (V)\neq \emptyset $.
			\end{definition}
It is a standard fact that it is equivalent to one the following properties : 
\begin{itemize}
			\item[(1)] there is a $\phi_\R^{\theta}-$dense orbit in $\Omega$.
			\item[(2)] for all open nonempty subsets $\widetilde{U},\widetilde{V}\subset \widetilde{\Omega}(X)$, there exists $t_n\rightarrow +\infty$ such that for every $n\geq 1$,
			$\G \widetilde{U} \cap \phi_{t_n}^\theta ( \widetilde{V})\neq \emptyset $.
			\item[(3)] for all open nonempty subsets $\widetilde{U},\widetilde{V}\subset \widetilde{\Omega}$, there exists $t_n\rightarrow +\infty$ such that for every $n\geq 1$, there exists $\g_n \in \G$ with
			$\g_n \widetilde{U} \cap \phi_{t_n}^\theta ( \widetilde{V})\neq \emptyset $.
			\end{itemize}

The equivalence between the definition and property (1) can be found in  \cite[Proposition 3.5]{eberlein1972flow}. The others equivalences are straightforward.

\subsection{Limit set, limit cone of Zariski dense subgroup}

\begin{center}
\fbox{
\begin{minipage}{0.7\textwidth}
In the remaining parts of this paper, $\G \subset G$ is a Zariski dense semigroup of $G$.
\end{minipage}}
\end{center}

\begin{definition}\label{defin_limit_set}
A point $\eta \in \cF(X)$ is a \emph{limit point} if there exists a sequence $(\g_n)_{n\geq 1}$ in $\G$ such that $(\g_n [f_0(a^+)])_{n\geq 1}$ converges in $\cF(X)$ towards $\eta$.

The \emph{limit set} of $\G$, denoted by $L_+(\G)$, is the set of limit points of $\G$. 
It is a closed subset of $\ccF(X)$. 

Denote by $L_-(\G)$ the limit set of $\G^{-1}$ and finally let $L^{(2)}(\G)= \Big( L_+(\G)\times L_-(\G) \Big) \cap \mathcal{F}^{(2)}(X) $ the subset of $\cF^2(X)$ in general position.
\end{definition}

\begin{lemme}[\cite{benoist1997proprietes} Lemma 3.6 ] \label{lemme3.6-benoist1}
The set of pairs of attractive and repulsive points of loxodromic elements of $\G$ is dense in $L_+(\G)\times L_-(\G)$.
\end{lemme}

\begin{definition}
We denote by $\widetilde{\Omega}(X)$ the subset of \emph{non-wandering Weyl chambers}, defined through the Hopf parametrization by : 
$$\widetilde{\Omega}(X):= \cH^{-1}(L^{(2)}(\G)\times \fa).$$
This is a $\G-$invariant subset of $\mathcal{W}(X)$.
When $\G$ is a subgroup, we denote by $\Omega(X):=\G \backslash \widetilde{\Omega}(X)$ the quotient space.
\end{definition}

Conze and Guivarc'h proved in \cite[Theorem 6.4]{conze-guivarch-densite2000}, the existence of dense $\fa-$orbits in $\widetilde{\Omega}(X)$. 
By duality, it is equivalent to topological transitivity of left $\G-$action on $\widetilde{\Omega}(X)/AM \simeq L^{(2)}(\G)$.
We propose a new simpler proof of this result adapting the one for negatively curved manifolds of Eberlein \cite{eberlein1972flow}.
\begin{theorem}[\cite{conze-guivarch-densite2000}]\label{transitivity of Furstenberg boundary}
 For any open nonempty subsets $\cU^{(2)}, \cV^{(2)} \subset L^{(2)}(\G)$ there exists $g\in \G$ such that $g \cU^{(2)} \cap \cV^{(2)}\neq \emptyset$. 
\end{theorem}

\begin{proof}
Whithout loss of generality, we assume that $\cU^{(2)}=\cU_+\times \cU_-$ and $\cV^{(2)}=\cV_+\times \cV_-$ where $\cU_+,\cV_+$ (resp. $\cU_-,\cV_-$) are open nonempty subsets of $L_+(\G)$ (resp. $L_-(\G)$).  

We choose an open set $W^{(2)}=W_+\times W_- \subset L^{(2)}(\G)$ so that $\cV_+$ and $W_-$ (resp. $W_+$ and $\cU_-$) are opposite. 
Such a choice is always possible.
If $\cV_+$ and $\cU_-$ are opposite, we can take $W^{(2)}=\cV^{(2)}$.
Otherwise, by taking $\cU^{(2)}$ and $\cV^{(2)}$ smaller, we can always assume that the subset of points in $L_+(\G)$ (resp. $L_-(\G)$) in general position with $\cU_-$ (resp. $\cV_+$) is non empty. Then we choose a suitable opposite pair of open nonempty subsets $W_+\times W_- \subset L_+(\G) \times L_-(\G)$.   

Since $W_+\times\cU_- \subset L^{(2)}(\G)$, then, by Lemma \ref{lemme3.6-benoist1}, there are loxodromic elements in $\G$ with attractive point in $W_+$ and repulsive point in $\cU_-$.
By Lemma \ref{lemme_carac_lox}, such a loxodromic element $\g_1$ contracts points that are in general position with $\g_1^-\in \cU_-$ towards $\g_1^+ \in W_+$.
Apply now statement (iv) of Lemma \ref{lemme_carac_lox}, to loxodromic element $\g_1$, with $W_-$ in general position with $\g_1^+$ and $U_-$ containing $\g_1^-$.
Hence for any $n$ large enough, $\g_1^n\cU^{(2)}\cap W^{(2)}\neq \emptyset$.

We take an open subset $\cW^{(2)}$ of $\g_1^n\cU^{(2)}\cap W^{(2)}$ of the form $\cW^{(2)}=\cW_+\times \cW_-$.
Then $\cV_+\times \cW_- \subset \cV_+ \times W_- \subset L^{(2)}(\G)$. 
Likewise, we choose a loxodromic element $\g_2 \in \G$ so that $\g_2 \cW^{(2)} \cap \cV^{(2)} \neq \emptyset $.
Then $$ \Big( \g_2\g_1^n\cU^{(2)}\cap \g_2W^{(2)} \Big) \cap \cV^{(2)} \supset \g_2 \cW^{(2)} \cap \cV^{(2)}\neq \emptyset.$$

Finally, the element $g= \g_2 \g_1^n$ satisfies $ g \cU^{(2)} \cap \cV^{(2)} \neq \emptyset$ 
\end{proof}

The theorem  below describes the set of directions $\theta \in \fa^+_1$ for which we will show that $\phi_t^\theta$ is topologically mixing. 

	\begin{theorem}[\cite{benoist1997proprietes}]\label{theorem_cone_limite}
We define the  \emph{limit cone} of $\G$ by, 	$\cC(\G):= \overline{ \underset{\g\in \G}{\bigcup} \R \lambda(\g)}$. We have 
	$$\cC(\G)=\underset{n\geq 1}{\bigcap} \overline{\underset{\g\in \G}{\underset{ \Vert \g\Vert\geq n}{\bigcup}} \R \mu(\g)}.$$  
	Moreover this set is closed, convex, of non-empty interior.
	\end{theorem}

	\subsection{Topological transitivity properties}

Recall the definition of the subset of non-wandering Weyl chambers $\widetilde{\Omega}(X)= \cH^{-1}(L^{(2)}(\G)\times \fa).$ 

\begin{prop}\label{prop_transitivité}
	Let $\theta \in \fa^{++}$.
	If the flow $(\Omega(X),\phi_\theta^t)$  is topologically transitive then $\theta \in \overset{\circ}{\cC}(\G)$.
\end{prop}

\begin{proof}
We assume that the dynamical system $(\Omega(X),\phi_\theta^t)$ is topologically transitive i.e. there exists a dense orbit.
Let $x\in\Omega(X)$ be a point of $\phi_\theta^t-$dense orbit and choose $g_x\in G$ a lift of $x$ in $G$. 

By density of $(\phi_\theta^t(x))_{x\in \R}$, for any $yM\in \widetilde{\Omega}(X)\subset G/M$, 
there exists $t_n\rightarrow +\infty$, $\varepsilon_n \rightarrow id_G$, $m_n\in M$ and $\g_n\in \G$ so that
$$\phi_\theta^{t_n}(g_x)=g_x e^{t_n \theta}=\g_n y \varepsilon_n m_n.$$ 

In particular, since the element $y=g_xe^{-v}$ belongs to $\widetilde{\Omega}(X)$ for all $v\in \fa$, there exists $t_n\in \R, \epsilon_n\tv \Id_G$, $m_n\in M$ and $\g_n\in\G$ such that 
\begin{equation}\label{eq4}
g_x e^{t_n \theta}=\g_n g_x e^{-v} \varepsilon_n m_n.
\end{equation}
For every $n\geq 1$ we set 
$\varepsilon_n':= g_x e^{-v} \varepsilon_n e^{v}g_x^{-1}$. 
The sequence $(\varepsilon_n')_{n\geq 1}$ converges towards $id_G$, and we have:

\begin{equation}\label{eq5}
g_x e^{v+t_n \theta} m_n^{-1}g_x^{-1}=\g_n \varepsilon_n'.
\end{equation}
For every $n\geq 1$ we set  $g_n:=g_x e^{v+t_n \theta} m_n^{-1}g_x^{-1}$.

Thanks to Lemma \ref{lemme_cartan_produit}, we deduce the following estimates
\begin{align*}
\mu(g_n) &\in v+ t_n \theta + C_{g_x}+C_{g_x^{-1}} \\
\mu(\g_n \varepsilon_n')  &\in \mu(\g_n)+ C_{\varepsilon_n'}.
\end{align*}
Therefore, $\mu(\g_n)$ is at bounded distance to $v+t_n\theta $ and  by Theorem \ref{theorem_cone_limite}, $\theta$ must lie in the (closed) limit cone. We now show that $\theta$ cannot be in its boundary. For this we need to study more carefully the Jordan projection of $g_n\epsilon_n'^{-1}$. \\

By definition, $\lambda(\g_n)$ belongs to the limit cone. By computing $\lambda(g_n \epsilon_n'^{-1})$, we will show that $\lambda(\g_n)$ also lies in a uniform (with respect to $v$) neighborhood of $v+t_n\theta$.  Finally, choosing $v$ far enough will force $\theta$ to be in the interior of the limit cone. \\

First of all, let us show that $g_n \epsilon_n'^{-1}$ is a loxodromic element. 
Since by hypothesis $\theta$ is in the interior of the Weyl chamber $\fa^+$,
there exists $n_0\in \N$ so that for $n\geq n_0$ large enough, $\lambda(g_n)=v+t_n \theta \in \fa^{++}$.
Hence $g_n$ is loxodromic and $(g_n^+,g_n^-)=(g_x \eta_0,g_x \check{\eta}_0)$ for all $n\geq n_0$.
We choose $0<r\leq \frac{1}{7} d(g_x \eta_0,g_x \check{\eta}_0) $.

We apply Lemma \ref{lemme_rem_Cagri_lox} on the loxodromic elements $ \big(g_x e^\theta g_x^{-1} \big)^k $.
There is a sequence of $\rho_k \rightarrow 0$ so that $ \big(g_x e^\theta g_x^{-1} \big)^k $ is $(r,\rho_k)-$loxodromic. 
Then for any $n\geq n_0$, $g_n$ is the product of a $(r,\rho_{k_n})-$loxodromic element and a loxodromic element of the form $g_x e^{v_n}m_n^{-1}g_x^{-1}$, where $v_n \in \fa^+$ is bounded, and with $k_n \rightarrow + \infty$. 
Since $g_x e^{v_n}m_n^{-1}g_x^{-1}$ and $\big(g_x e^\theta g_x^{-1} \big)^k $ have the same attractive and repulsive point in $\mathcal{F}(X)$, we deduce that $g_n$ is $(r,\rho_{k_n})-$loxodromic for $n\geq n_0$. 
Take now $\rho_n'=\max(\rho_n,\frac{1}{2}\Vert \varepsilon'_n -id_G\Vert)$.
Then there exists $n_1$ so that for $n\geq \max(n_0,n_1)$, then $0<\rho'_n \leq r$, and $g_n$ is $(r,\rho'_n)-$loxodromic.  
Corollary \ref{corollaire_produit_lox} shows that $g_n \varepsilon_n'^{-1}$ is $(2r, 2\rho'_n)-$loxodromic for $n$ large enough, and $(g_n \varepsilon_n'^{-1})^+ \in B(g_x\eta_0,\rho'_n)$.

Using Fact \ref{fait_lox}, we compute $\lambda(g_n\varepsilon_n'^{-1})$:
\begin{align*}
\lambda(g_n \varepsilon_n'^{-1})&=\sigma(g_n \varepsilon_n'^{-1},(g_n \varepsilon_n'^{-1})^+)\\
&= \sigma(g_n, \varepsilon_n'^{-1} (g_n \varepsilon_n'^{-1})^+)+ \sigma(\varepsilon_n'^{-1} ,(g_n \varepsilon_n'^{-1})^+) \\
&= \sigma(g_n, g_x\eta_0)  \\
& \quad \quad \quad +\Big( \sigma\big(g_n, \varepsilon_n'^{-1} (g_n \varepsilon_n'^{-1})^+\big)
-\sigma(g_n,g_x\eta_0) \Big) \\
&\quad \quad \quad + \sigma(\varepsilon_n'^{-1} ,(g_n \varepsilon_n'^{-1})^+) .\\ 
\end{align*}
Remark that,
$\sigma(g_n, g_x \eta_0)=\lambda(g_n)=v+t_n \theta. $
hence
\begin{equation}
\lambda(g_n \varepsilon_n'^{-1})-(v+t_n\theta)= \Big( \sigma\big(g_n, \varepsilon_n'^{-1} (g_n \varepsilon_n'^{-1})^+\big)
-\sigma(g_n,g_x\eta_0) \Big) + \sigma(\varepsilon_n'^{-1} ,(g_n \varepsilon_n'^{-1})^+).
\end{equation}

We analyze separately the two terms of the right hand side of the last equality. 

For the last term, by Lemma \ref{lemme_cartan_produit} (ii) 
$$\Vert \sigma(\varepsilon_n'^{-1} ,(g_n \varepsilon_n'^{-1})^+) \Vert \leq  C_{\varepsilon_n'} .$$ 

Now we will bound, independently of $v$, the term $\sigma(g_n, \varepsilon_n'^{-1} (g_n \varepsilon_n'^{-1})^+)
-\sigma(g_n,g_x \eta_0)$.

Let $\alpha\in \Pi$ be a simple root and consider the proximal representation of $G$ associated to $\alpha$.
By Lemma \ref{lemme_repr_cartan_jordan_iwasawa} (b)(iii), for any $\xi \in \varepsilon_n'^{-1} B(g_x \eta_0,2\rho'_n)$, there exists a non zero vector $v_\xi \in V_\alpha$ so that 
$$\chi_{\max,\alpha}(\sigma(g_n,\xi))=\log \frac{\Vert \rho_\alpha (g_n)v_\xi \Vert }{\Vert v_\xi \Vert}.$$
Let $\xi= \varepsilon_n'^{-1} (g_n \varepsilon_n'^{-1})^+$ and consider a unitary vector $v_\xi\in V_\alpha$.
Since $\xi$ is in a $3\rho'_n-$neighbourhood of $g_n^+$, we write
$v_\xi=v_+ + v_<$  where $v_+ \in V_{+}(\rho_\alpha( g_n))$ and $v_<\in V_{-}(\rho_\alpha( g_n))$ with $\Vert v_+ \Vert>1- 3\rho'_n $.
Then
$$\frac{\rho_\alpha(g_n)}{\lambda_1(\rho_\alpha(g_n))}(v_\xi)
= v_+ + \frac{\rho_\alpha(g_n)}{\lambda_1(\rho_\alpha(g_n))}(v_<) $$ 
By the triangle inequality,
$$
\Vert v_+ \Vert - \Big\Vert \frac{\rho_\alpha(g_n)}{\lambda_1(\rho_\alpha(g_n))}(v_<) \Big\Vert
\leq 
\frac{\Vert \rho_\alpha(g_n) v_\xi \Vert }{\lambda_1(\rho_\alpha(g_n))} 
\leq 
\Vert v_+\Vert +
\Big\Vert \frac{\rho_\alpha(g_n)}{\lambda_1(\rho_\alpha(g_n))}(v_<) \Big\Vert.
$$
Hence 
$$
1-3\rho'_n - \Big\Vert \frac{\rho_\alpha(g_n)}{\lambda_1(\rho_\alpha(g_n))}(v_<) \Big\Vert
\leq 
\frac{\Vert \rho_\alpha(g_n) v_\xi \Vert }{\lambda_1(\rho_\alpha(g_n))} 
\leq 
1 +
\Big\Vert \frac{\rho_\alpha(g_n)}{\lambda_1(\rho_\alpha(g_n))}(v_<) \Big\Vert.
$$

The eigenvalues of $ \frac{\rho_\alpha(g_n)}{\lambda_1(\rho_\alpha(g_n))}$ restricted to $X_-(g_n)$ are 
$\exp( \chi_\alpha (\lambda(g_n))-\chi_{max, \alpha}(\lambda(g_n))  )$, where $\chi_\alpha\neq \chi_{max,\alpha}$ is a restricted weight of $\Sigma(\rho_\alpha)$.
They converge to zero and these endomorphisms are all diagonalisable.
Hence, 
$$\Big\Vert \frac{\rho_\alpha(g_n)_{\vert X_-(g_n)}}{\lambda_1(\rho_\alpha(g_n))}\Big\Vert \underset{n +\infty}{\rightarrow}0.$$

Taking the logarithm and the upper bound of $\frac{\Vert \rho_\alpha(g_n) v_\xi \Vert }{\lambda_1(\rho_\alpha(g_n))} $ and its inverse, we obtain for $n$ large enough,
$$\Vert \sigma(g_n, \xi)
-\sigma(g_n,g_x \eta_0) \Vert \leq 3 \rho'_n + \sup_{\alpha\in \Pi}\Big\Vert \frac{\rho_\alpha(g_n)_{\vert X_-(g_n)}}{\lambda_1(\rho_\alpha(g_n))}\Big\Vert. $$

Finally, for any $v\in \fa$, there exists $t_n \rightarrow +\infty$, $\varepsilon_n' \rightarrow id_G$, so that for any $n$ large enough,
\begin{equation}\label{eq_trans}
\Vert \lambda(\g_n) - (v+t_n \theta) \Vert \leq  3 \rho'_n + \sup_{\alpha\in \Pi}\Big\Vert \frac{\rho_\alpha(g_n)_{\vert X_-(g_n)}}{\lambda_1(\rho_\alpha(g_n))}\Big\Vert + C_{\varepsilon'_n} .
\end{equation}
The three terms converge to zero when $n\rightarrow +\infty$, so that, for $n$ large enough, $\lambda(\g_n)-(v+t_n \theta)$ is uniformly bounded.

To conclude, recall that the limit cone is the smallest closed cone containing all the Jordan projections of $\G$.
Hence, this implies that for $n$ large enough, the distance $d(v+t_n\theta, \cC(\G))$ is uniformly bounded. 
Now, assume by contradiction that $\theta$ is in the boundary of $\cC(\G)$.
Let $H$ be a supporting hyperplane of the convex $\cC(\G)$ at $\theta$ and $H^+$ the half space not containing $\cC(\G)$. 
Pick $v\in H^+$, whose  distance to $\cC(\G)$ is large. 
Then $d(v+\R^+\theta,\cC(\G)) = d(v+\R^+\theta,H) = d(v,H)$ is also large. 
This is contradictory with inequality (\ref{eq_trans}).

Hence, topological transitivity of the dynamical system $(\Omega(X),\phi_\theta^t)$ implies that $\theta \in \overset{\circ}{\cC}(\G)$.
\end{proof}

\section{Topological mixing}
	Recall the definition of topological mixing.
		\begin{definition}\label{definition_topological_mixing}
Fix a direction $\theta\in \fa_1^{++}$.
		The Weyl chamber flow $\phi_\R^{\theta}$ is \emph{topologically mixing} on $\Omega(X)$ if for all open subsets $U,V \subset \Omega(X)$, there exists $T>0$ such that for all $t\geq T$, we have $U\cap \phi_t^\theta (V)\neq \emptyset $.
		\end{definition}
It will be sometimes more convenient to make proofs in the cover $	 \widetilde{\Omega}(X)$, where the topological mixing takes  the following form : for all open subsets $\widetilde{U},\widetilde{V}\subset \widetilde{\Omega}(X)$, there exists $T>0$ such that for all $t\geq T$ there exists $\g_t \in \G$ with $\g_t \widetilde{U} \cap \phi_t^\theta ( \widetilde{V})\neq \emptyset $.

%
%

	\subsection{Non-arithmetic spectrum}
Denote by $\G^{lox}$ the set of loxodromic elements of $\G$. 
In \cite{dalbo2000feuilletage} Dal'bo introduced the notion of \emph{non-arithmetic spectrum } for subgroup of $Isom(H^n)$, meaning that the length spectrum of such a group is not contained in a discrete subgroup of $\R$. 

We generalize this definition for isometry group in higher rank:
\begin{definition}
We say that $\G$ has  \emph{non-arithmetic spectrum} if the length spectrum $\lambda(\G^{lox})$ spans a dense subgroup of $\fa$.   
\end{definition}

\begin{prop}\label{prop_benoistII}
Every Zariski dense semigroup $\G$ contains loxodromic elements, strong $(r,\varepsilon)$-Schottky Zariski dense semigroups and has non-arithmetic spectrum.
\end{prop}
\begin{proof}
For a general semisimple, connected, real linear Lie group, Benoist proves in \cite[Proposition 0]{benoist2000proprietes} that when $\G$ is a Zariski dense semigroup of $G$, then the additive group generated by the full length spectrum $\lambda(\G)$ is dense in $\fa$.
Thus, this Proposition implies that  Zariski dense semigroups containing only loxodromic elements have non-arithmetic length spectrum. 
In particular, strong $(r,\varepsilon)$-Schottky Zariski dense semigroups have  non-arithmetic length spectrum. 
Finally, the existence of Zariski dense Schottky semigroups in Zariski dense subgroups of $G$ follows from \cite[Proposition 4.3 for $\theta=\Pi$]{benoist1997proprietes}.
\end{proof}

Prasad and Rapinchuk \cite[Theorem 2]{prasad2005zariski} prove that every Zariski dense semigroup of $G$ contains a finite subset $F$ such that $\lambda(F)$ generates a dense subgroup of $\fa$.

\subsection{A key proposition for mixing}
The following proposition is the technical point for proving the topological mixing of the Weyl chamber flow. Roughly, it shows that among elements of $\G$ which do not move too much a flat, (ie. $(\g_t^+,\g_t^-)\in \mathcal{U}^{(2)} $) for any given $x\in \fa$, we can find an element which send $0$ to $x+\theta t$ for large $t$   (ie.   $\lambda(\g_t)\in B(x+t \theta,\eta) $)

\begin{prop}\label{prop-moving along fa}
Fix $\theta\in \fa_1^{++}$ in the interior of the limit cone $\mathcal{C}(\G)$.

Then for every nonempty open subset $\mathcal{U}^{(2)}\subset L^{(2)}(\G)$, for all $x\in \fa$ and $\eta >0$ there exists $T>0$ such that for all $t\geq T$ there exists a loxodromic element $\g_t \in \G$ with
\begin{equation}\label{bouger selon fa}
  \left\{
      \begin{aligned}
        (\g_t^+,\g_t^-)\in \mathcal{U}^{(2)} \\
        \lambda(\g_t)\in B(x+t \theta,\eta)  \\
      \end{aligned}
    \right.
\end{equation}  
\end{prop}
We will need the following classical density lemma, see for example \cite[Lemma 6.2]{benoist2000proprietes}. \begin{lemme}\label{lemme_densite_cone}
Let $V$ be a real vector space of finite dimension. 
Let $l_0,l_1,...,l_t$ be vectors of $V$ and $\eta>0$. Set $$ L:= \underset{0\leq i \leq t}{\sum} \R_+ l_i, \; 
M:= \underset{0\leq i \leq t}{\sum} \mathbb{Z} l_i, \; \text{and}\;
M_+:= \underset{0\leq i \leq t}{\sum} \mathbb{N} l_i. $$
Assume that  $M$ is $\eta$-dense in $V$.
Then there exists $v_0\in V$ such that $M_+$ is $\eta$-dense in $v_0+L$.
\end{lemme}
Remark that if $M_+$ is $\eta$-dense in $v_0+L$ then it is $\eta$-dense in $v+L$ for every $v\in v_0+L$.

The following lemma is a consequence of \cite[Proposition 4.3]{benoist1997proprietes}.
	\begin{lemme}\label{lemme_schottky_theta}
	For all $\theta$ in the interior of the limit cone $\mathcal{C}(\G)$, there exists a finite set $S\subset \G$, a positive number $r>0$ and $\varepsilon_n \underset{+\infty}{\rightarrow} 0$ such that
		\begin{itemize}
		\item[(i)] $\theta$ is in the interior of the convex cone $L(S):= \underset{\g\in S}{\sum}\R_+ \lambda(\g)$,
		\item[(ii)] the elements of $\lambda(S)$ form a basis of $\fa$, 
		\item[(iii)] for all $n\geq 1$, the family $S_n:= (\g^n)_{\g\in S}$ spans a Zariski-dense strong $(r,\varepsilon_n)-$Schottky semigroup of $\G$. 
		\end{itemize}
	\end{lemme}
	\begin{proof}
	Fix $\theta$ in the interior of $\mathcal{C}(\G)$.
	
	Let us now construct a family of $r_G$ open cones in the limit cone $\mathcal{C}(\G)$.
	We consider a affine chart of $\mathbb{P}(\fa)$ centered in $\R\theta$. 
	Since $\R \theta$ is in the open set $\mathbb{P}(\overset{\circ}{\mathcal{C}}(\G))$, it admits an open, polygonal, convex neighborhood with $r_G$ distinct vertices centered in $\R \theta$ and included in $\bP (\overset{\circ}{\cC}(\G))$.
	We denote by $p:=(\R p_i)_{1\leq i\leq r_G}$ the family of vertices of that convex neighbourhood, $\mathscr{H}_p$ its convex hull.
	Without loss of generality we can assume that there exists $\delta_0 >0$ so that the $\delta_0-$neighbourhood of  $\mathscr{H}_p$,  $\mathcal{V}_{\delta_0}(\mathscr{H}_p)$ is included in $\mathbb{P}(\overset{\circ}{\mathcal{C}}(\G))$.
	
	For any $\delta>0$, we denote by $\mathcal{V}_{\delta}(\partial \mathscr{H}_p)$ the $\delta-$neighborhood of the boundary $\partial \mathscr{H}_p$.
	Choose $0<\delta\leq \inf\big(\delta_0, \frac{1}{3}d(\R \theta, \partial \mathscr{H}_p) \big)$	so that $\R \theta \in \mathscr{H}_p\setminus \mathcal{V}_{\delta}(\mathscr{H}_p)$.
	
	Denote by $L_p\subset \overset{\circ}{\mathcal{C}}(\G)$ (resp. $\mathcal{V}_{\delta}(\partial L_p$)) the closed (resp. open) cone whose projective image is $\mathscr{H}_p$ (resp. $\mathcal{V}_{\delta}(\partial \mathscr{H}_p) $). 
	For all $1\leq i \leq r_G$, denote by $(\Omega_i)_{1\leq i \leq r_G}$ the family of open cones such that $\mathbb{P}(\Omega_i):= B_{\mathbb{P}(\fa)}(p_i,\delta)$.
	
	By \cite[Proposition 4.3]{benoist1997proprietes} applied to the finite family of disjoint open cones $(\Omega_i)_{1\leq i \leq r_G}$  there exists $0<\varepsilon_0 \leq r $, a generating set $S:=\{\g_i\}_{1\leq i \leq r_G}\subset \G$ of a Zariski dense $(r,\varepsilon)-$Schottky semigroup such that for all $1\leq i \leq r_G$ the Jordan projection $\lambda(\g_i)$ is in $\Omega_i$.
	By Lemma \ref{fait_rem_Cagri}, for any $n\geq 1$, the elements of $S_n$ are $(r,\varepsilon_n)-$loxodromic.
	Thus, for $n$ large, condition (iii) holds.
	By construction, $\lambda(S)$ form a  family of $r_G$ linearly independent vectors of $\fa$ hence (ii) holds.
	Set $L(S):= \underset{\g\in S}{\sum}\R_+ \lambda(\g)$.
	The construction of $L_p$ and $ \mathcal{V}_{\delta}(\partial L_p)$ implies that $\theta \in L_p \setminus \mathcal{V}_{\delta}(\partial L_p)$. 
	Since $\lambda(\g_i)\in \Omega_i \subset  \mathcal{V}_{\delta}(\partial L_p)$ for all $1\leq i \leq r_G$,  the boundary of the cone $\partial L(S) \subset\mathcal{V}_{\delta}(\partial L_p)$.
	Hence $L_p \setminus \mathcal{V}_{\delta}(\partial L_p)\subset \overset{\circ}{L}(S)$ and finally, condition (i) holds, $\theta$ is in the interior of the cone $L(S)$.
	\end{proof}

Let us give a proof of the key Proposition.

\begin{proof}[Proof of Proposition \ref{prop-moving along fa}.]
We fix a point $\theta$ in the interior of $\mathcal{C}(\G)$, an open, nonempty set $\cU =\cU^+\times \cU^-\subset L^{(2)}(\G)$, a point $x\in \fa$ and $\eta >0.$

Consider $S$ as in the previous Lemma \ref{lemme_schottky_theta}.
Denote by $\G_{n}$ the semigroup spanned by $S^n$. 
 
 By Lemma \ref{prop - lemme 1.4 de Benoist ii }, one can pick $h \in \G^{lox}$ such that $(h^+,h^-) \in  \mathcal{U}^{(2)}\setminus (\g_1^-,\g_{r_G}^+)$. 
Choose $r>0$ so that
 $$r\leq \inf\bigg(\rho, \frac{1}{6}d(h^+,h^-),\frac{1}{6} d(\g_{r_G}^+,h^-), \frac{1}{6}d(h^+,\g_1^-) \bigg).$$ In particular, Proposition \ref{prop - lemme 3.6 de Benoist ii }  holds for elements of the form $h \g_{r_G}^{n_{r_G}}g \g_1^{n_1}h$ where $g\in \G_{n}$. 
 
Choose $0<\varepsilon \leq r$ small enough so that 
\begin{equation}\label{choix de r et eps}
  \left\{
      \begin{aligned}
		(3 r_G +2)C_{r,\varepsilon}\leq \eta/2 \\
		B(h^+,\varepsilon)\times B(h^-,\varepsilon)\subset \mathcal{U}^{(2)}  \\    \end{aligned}
    \right.
\end{equation} 
where $(C_{r,\varepsilon})_{\varepsilon\geq 0}$ are constants given by the Proposition.   

We use Lemma \ref{fait_rem_Cagri} and choose $n$ large so that $h^n, S^n$ are $(r,\varepsilon_n)-$loxodromic elements with $\varepsilon_n\leq \varepsilon$. 

By Proposition \ref{prop_benoistII}, the subgroup generated by $\lambda(\G_{n})$ is dense in $\fa$. 
By Lemma \ref{lemme_densité} applied to $\lambda(\G_{n})$, there exists a finite subset $F\subset \G_{n}$ containing at most $2r_G$ elements so that $\lambda(S^n)\cup \lambda(F)$ spans a $\eta/2-$dense subgroup of $\fa$.
We denote by $l$ the number of elements in $S':=S^n\cup F$ and we enumerate  the elements of $S^n\cup F$  by $(g_1,..., g_l)$, where $g_1:=\g_1^n$ and $g_l:= \g_{r_G}^n$. 
A crucial fact is that $l\leq 3 r_G$ is  bounded independently of $\lambda(\G_{n})$. 
 
The additive subgroup generated by $\lambda(S')$ is $\eta/2-$dense in $\fa$. Furthermore, $\theta$ is still in the interior of the convex cone $L(S'):=\underset{g\in S'}{\sum}\R_+ \lambda(g)$ by (i). 
Lemma \ref{lemme_densite_cone} gives the existence of $v_0\in \fa$ such that $M_+(S'):= \underset{g\in S'}{\sum}\mathbb{N} \lambda(g)$ is $\eta/2-$ dense in $v_0+ L(S')$.
 
The interior of $L(S')$ contains $\theta$.
Hence for any $v\in \fa$, the intersection $\big(v+\R_+ \theta \big) \cap \big(v_0+ L(S')\big)$ is a half line. 

Consider such a half line $ x-\nu(h^n,g_l,...,g_1,h^n)-2\lambda(h^n) + \theta [T,+\infty)$ contained in $v_0+L(S')$, for some $T\in \R$.
For all $t\geq T$, there exists $n_t:=(n_t(1),...,n_t(l)) \in \mathbb{N}^l$ such that
\begin{equation}\label{eq - densite}
\bigg\Vert \overset{l}{\underset{i=1}{\sum}} n_t(i)\lambda(g_i) 
-x +\nu(h^n,g_l,...,g_1,h^n)+ 2\lambda(h^n) -\theta t \bigg\Vert \leq \eta/2 .
\end{equation}
Furthermore, Proposition \ref{prop - lemme 3.6 de Benoist ii } applied to $\g_t:=h^n g_l^{n_t(l)} ... g_1^{n_t(l)}h^n$ gives
\begin{equation}\label{eq - proposition1}
\bigg\Vert \lambda( \g_t ) -\overset{l}{\underset{i=1}{\sum}} n_t(i) \lambda(g_i)  - 2\lambda(h^n)- \nu(h^n,g_l,...,g_1,h^n) \bigg\Vert \leq  (l+2)C_{r,\varepsilon} 
\end{equation} 
and $(\g_t^+, \g_t^-) \in B(h^+,\varepsilon)\times B(h^-,\varepsilon)\subset \mathcal{U}^{(2)}$ by (\ref{choix de r et eps}).

Finally, we have $(3r_G+2)C_{r,\varepsilon}\leq \eta/2$ by the choice of $n$, $S^n$, $h^n$. Once again, remark it is necessary for $l$ to be  bounded independently of $\G$ and $n$. We get the following bound using the triangle inequality,
\begin{equation}\label{eq - proposition2}
\Vert \lambda( \g_t ) - x - \theta t\Vert \leq  \eta. 
\end{equation}
This concludes the proof.
\end{proof}	

Note that it is possible to use the density \cite[Theorem 2]{prasad2005zariski} of Prasad and Rapinchuk instead of our density Lemma \ref{lemme_densité}.
Start by following our proof, choose $S\subset \G$ as in the Lemma \ref{lemme_schottky_theta}.
Use now Prasad and Rapinchuk's density Theorem, there is a finite subset $F$ of the semigroup generated by $S$ so that $\langle\lambda(F)\rangle$ is dense in $\fa$. 
Remark that for any $n\in \mathbb{N}$, the subset $S''_n:=F^n\cup S^n$ is finite, has at most $\vert F \vert+r_G $ elements and the subgroup generated by $\lambda(S''_n)$ is also dense in $\fa$.   
It suffices then to follow the end of the proof by taking $S'=S''_n$ for $n$ large enough so that $S'$ is a $(r,\varepsilon_n)-$Schottky semigroup with
$(\vert F \vert+r_G +2)C_{r,\varepsilon_n}\leq \eta/2$.

\subsection{Proof of the main Theorem \ref{th-main}}
We end the proof of the main theorem with Proposition \ref{prop-moving along fa} and Theorem \ref{transitivity of Furstenberg boundary}.
\begin{proof}[Proof of Theorem \ref{th-main}]
If  ($\Omega(X), \phi_t^\theta)$ is topologically mixing, it is in particular topologically transitive. Therefore by Proposition \ref{prop_transitivité} if ($\Omega(X), \phi_t^\theta)$ is topologically mixing $\theta$ is in the interior of the limit cone.\\

Let us prove that if $\theta\in \overset{\circ}{\cC}(\G) \cap \fa_1^{++}$ then ($\Omega(X), \phi_t^\theta)$ is topologically mixing.

Let $\widetilde{U},\widetilde{V}$ be two open subsets of $\widetilde{\Omega}(X)$.
Without loss of generality, we can assume that $\widetilde{U}=\cH^{-1}(\cU^{(2)} \times B(u,r))$ (resp. $\widetilde{V}=\cH^{-1}(\cV^{(2)} \times B(v,r))$) where $\cU^{(2)}$ and $\cV^{(2)}$ are open subsets of $L^{(2)}(\G)$, and  $B(u,r), B(v,r)$  open balls of $\fa$.

Recall that for all $g\in \G$, using Hopf coordinates 
\begin{equation}\label{dans la preuve du thm}
  \left\{
      \begin{aligned}
        \cH^{(2)}\big(g (\cU^{(2)})\times B(u,r) \big)= g\cU^{(2)} \\
        \cH (\phi^\theta_t(\cV^{(2)})\times B(v,r)) ) = \cV^{(2)}\times B(v+\theta t ,r)  \\
      \end{aligned}
    \right.
\end{equation}  

We begin by transforming the coordinates in $L^{(2)}(\G)$ to recover the setting of Proposition \ref{prop-moving along fa}.
By Theorem \ref{transitivity of Furstenberg boundary}, there exists $g\in \G$ such that $g\cU^{(2)}\cap \cV^{(2)}\neq \emptyset$.
For such an element $g\in \G$, the subset $g\cU^{(2)}\cap \cV^{(2)}$ is a  nonempty open subset of $L^{(2)}(\G)$.
Let $\mathcal{O}^{(2)}:= \mathcal{O}_+\times \mathcal{O}_- \subset g\cU^{(2)}\cap \cV^{(2)}$  be a nonempty open subset, such that  $r:=d(\overline{\mathcal{O}_+}, \overline{\mathcal{O}_-})>0$.

Remark that $g\widetilde{U} \cap \big( \cH^{(2)} \big)^{-1}(\mathcal{O}^{(2)})$ is open and non empty. 
Thus it contains an open box $\cH^{-1}(\mathcal{O}^{(2)}\times B(u',r') ) $ with $u'\in \fa$ and $r'>0$. 
Set $\eta:=\min (r,r')$.

By Proposition \ref{prop-moving along fa} applied to $\mathcal{O}^{(2)}$,  $x=v-u'\in\fa$ and $\eta>0$,
there exists $T>0$ such that for all $t\geq T$ there exists $\g_t \in \G$ with
\begin{equation}
  \left\{
      \begin{aligned}
        (\g_t^+,\g_t^-)\in \mathcal{O}^{(2)} \\
        \lambda(\g_t)\in B(v-u'+t \theta,\eta)  \\
      \end{aligned}
    \right.
\end{equation} 

Remark that every loxodromic element $\g \in \G$ fixes its limit points in $L^{(2)}(\G)$.
Thus for all such $\g\in \G$ with $(\g^+,\g^-) \in \mathcal{O}^{(2)}$, the subset $\g \mathcal{O}^{(2)} \cap \mathcal{O}^{(2)}$ is open and non empty (it contains $(\g^+,\g^-)$). Furthermore, $\lambda(\g)=\sigma(\g, \g^+)$ by Fact \ref{fait_lox}.
Hence 
\begin{equation}
  \left\{
      \begin{aligned}
         \g_t \mathcal{O}^{(2)} \cap \mathcal{O}^{(2)} \neq \emptyset \\
        u'+\sigma(\g_t,\g_t^+)\in B(v+t \theta,\eta)  \\
      \end{aligned}
    \right.
\end{equation}
The subset $\g_t g\tilde{U} \cap \big( \cH^{(2)} \big)^{-1}(\g_t \mathcal{O}^{(2)} \cap \mathcal{O}^{(2)})$ is open, non empty and contains the point of coordinates $(\g_t^+,\g_t^-, u'+\sigma(\g_t,\g_t^+))\in \cH^{-1}(\phi_t^\theta (\tilde{V}))$.
Finally, $\g_t g \tilde{U}\cap \phi_t^\theta(\tilde{V}) \neq \emptyset$, as $\tilde{U},\tilde{V}$ are arbitrary, it proves that $\phi_t^\theta$ is topological mixing.
\end{proof}

\section{Appendix: a density lemma}

The following density lemma is crucial for the proof of proposition \ref{prop-moving along fa}.
\begin{lemme}\label{lemme_densité}
Let $d\in \N$, let $V$ be a real vector space of dimension $d$.
For all $E\subset V$ that spans a dense additive subgroup of $V$, for all $\epsilon>0$,
 for any basis $B\subset E$ of $V$,  there exists a finite subset $F\subset E$ of at most $2d$ elements so that $B\cup F$ spans a $\epsilon-$dense additive subgroup of $V$. 
\end{lemme}

\begin{proof}
We show the lemma by induction. 

Let $\cE\subset \R^1=V$ be a subset that generates a dense additive subgroup of $\R$. Let $x\in \R$ a basis, ie. a non zero element. 
Any element $y$ in $\cE$ so that $\langle y,x \rangle$ is dense is a solution.
We assume that $\cE$ contains no such element.
Consider the quotient $\R/x\Z$ and $p \, : \R \tv \R/x\Z$ the projection.
The set $\cE$ projects to a infinite subset of $\R/x\Z$, therefore it has an accumulation point. Let $f_1\neq f_2\in \cE$ be two elements such that $|p(f_1)-p(f_2)|<\epsilon$.  Then $\langle x, f_1, f_2, \rangle$, generates a $\epsilon-$dense  additive subgroup of $\R$, the Lemma is proved for $dim(V)=1$, where $F = \{f_1,f_2\}$. 

Now consider a vector space $V$ of dimension $d$. Let $\cE$ be a a subset of $V$ such that $\overline{\langle \cE\rangle } = V$ and $\cB=(b_1,\dots, b_d) \subset E$ a basis of $V$. Without loss of  generality we suppose  that the basis is the standard basis and the norm is the sup norm : these only affect computations up to a multiplicative constant. 

Suppose that we have $f_1,f_2\in \cE$ such that the additive group $\langle f_1, f_2, \cB\rangle $ contains a non zero vector $u$ of norm $\| u\|\leq \epsilon/2$. We will show that it is enough to conclude and then prove the existence of such elements.  

Consider $V' = u^\perp$,  the decomposition $V= u\oplus V'$ and $p'$ the projection on $V'$. 
Let $\cE'=p'(\cE)$ and $\cB'$ a basis of $V'$ included in $p'(\cB)$. By induction, there is  a finite subset $\cF'\subset \cE'$  of at most $2(d-1)$ elements  such that $\langle \cF', \cB'\rangle$ generates an $\epsilon/2-$dense additive subgroup of $V'$. 
For all $f'\in \cF'$ there is  $f\in \cE$  and $\lambda_f\in \R$ such that $f'= f+\lambda_f u.$ A similar result holds for elements of $\cB'$. We denote by $\cF\subset \cE$ a choice of lifts for elements of  $\cF'$. We claim that the set $F = \cF \cup \{f_1, f_2\} \cup \cB$ generates a $\epsilon-$dense additive subgroup of $V$. 

Let $x\in V$, $x= x'+\lambda_x u$. By hypothesis, there is $(n_{f'})_{f'\in \cF'}\in \Z^{|\cF'|}$, and $(n_{b'})_{{b'}\cB'} \in \Z^{d-1}$ and $\alpha\in V'$ satisfying  $\|\alpha'\|< \epsilon/2$  such that : 
$$x' = \sum_{f'\in \cF'} n_{f'} f' + \sum_{b'\in \cB'} n_{b'} b' + \alpha'.$$
Therefore, 
$$x' = \sum_{f\in \cF} n_{f'} f +   \sum_{b\in \cB} n_{b} b + \Big( \sum_{f\in \cF} n_{f'}\lambda_f  +  \sum_{b\in \cB} n_{b} \lambda_b\Big)  u +  \alpha'  .$$
Finally we get : 
$$x =  \sum_{f\in \cF} n_{f'} f +   \sum_{b\in \cB} n_{b} b  +  [k]u  +(k-[k]) u +\alpha' $$
where $k =  \Big( \sum_{f\in \cF} n_{f'}\lambda_f  +  \sum_{b\in \cB} n_{b}  + \lambda_x\Big) $ and $[k]\in \Z$ denotes the integer part of $k$. 
The vector $ \sum_{f\in \cF} n_{f'} f +   \sum_{b\in \cB} n_{b} b  +  [k]u$ is in the additive group generated by $F$  and $|(k-[k]) u +\alpha| \leq \epsilon$. This proves the claim. 

To finish the proof we need to show that for any $\epsilon >0$, there are elements $f_1,f_2\in \cE$ such that $\langle f_1, f_2, \cB\rangle $ contains a non zero vector of norm less than $\epsilon$. 

Consider the natural projection $p \, : \R^d \tv \R^d/ \oplus_{k=1}^d \Z b_k$ into the torus $\R^d/ \oplus_{k=1}^d \Z b_k$.
If there is an element $f\in \mathcal{E}$ so that $p(\Z f)$ contains accumulation points, we choose $u$, non zero and small in $\langle \cB,f\rangle$.
We assume now that there is no such element in $E$.
Choose an integer $N$ so that $N>\frac{2 \sqrt{d}}{\varepsilon}$.
By the pigeon hole principle on $N^{d}+1$ distinct elements of $\mathcal{E}$, we deduce the existence of $f_1,f_2 \in \mathcal{E}$ with $0<\vert p(f_1-f_2) \vert<\frac{\epsilon}{2}$.
The unique representative of the projection $p(f_1-f_2)$ in the fundamental domain $\sum_{i=1}^d (0,1]b_i$ is a suitable choice for $u$.
Indeed, it is an element of the subgroup $\langle f_1,f_2, \cB \rangle$ and it is of norm at most $\frac{\epsilon}{2}$.

\end{proof}

\bibliographystyle{alpha}
\bibliography{bibliojabref}

\end{document}